\documentclass[11pt]{amsart}
\usepackage{amsmath,amsfonts,amsthm,amscd,amssymb,graphicx}

\newtheorem{theorem}{Theorem}[section]
\newtheorem{lemma}[theorem]{Lemma}
\newtheorem{proposition}[theorem]{Proposition}
\newtheorem{corollary}[theorem]{Corollary}

\newtheorem{remark}[theorem]{Remark}

\newcommand{\ip}{\int_{\mathbb{R}^+}}

\newcommand{\C}{\mathbb{C}}

\def\bb1{{1\!\!1}}

\def\CalL{\mathcal{L}}

\def\bu{\hat{u}}
\def\bv{\hat{v}}

\def\bU{{\hat{U}}}
\def\bV{{\hat{v}}}
\def \tu{\tilde{u}}
\def \tv{\tilde{v}}
\def\tU{{\widetilde{U}}}

\def\R{\Re e}
\def\I{\Im m}

\def\I{\Im m}
\def\sgn{\rm sgn}

\begin{document}
\title[Stability of boundary layers]
{Spectral stability of noncharacteristic isentropic 
Navier--Stokes boundary layers}
\author[Costanzino, Humpherys, Nguyen, and Zumbrun]{Nicola Costanzino, Jeffrey Humpherys, \\ Toan Nguyen, and Kevin Zumbrun}

\date{Last Updated:  June 22, 2007}

\thanks{ This work was supported in part by the National Science Foundation award numbers DMS-0607721 and DMS-0300487.}

\address{Department of Mathematics, Pennsylvania State
University, University Park, PA, 16802}
\email{costanzi@math.psu.edu}
\address{Department of Mathematics, Brigham Young University, Provo, UT 84602}
\email{jeffh@math.byu.edu}
\address{Department of Mathematics, Indiana University, Bloomington, IN 47402}
\email{nguyentt@indiana.edu}
\address{Department of Mathematics, Indiana University, Bloomington, IN 47402}
\email{kzumbrun@indiana.edu}

\begin{abstract}
Building on work of Barker, Humpherys, Lafitte, Rudd, and Zumbrun
in the shock wave case, we study stability of
compressive, or ``shock-like'',
boundary layers of the isentropic compressible
Navier--Stokes equations with $\gamma$-law pressure
by a combination of asymptotic ODE estimates and numerical Evans function
computations.
Our results indicate stability for $\gamma\in [1, 3]$ for all
compressive boundary-layers, independent of amplitude, save for 
inflow layers in the characteristic limit (not treated).
Expansive inflow boundary-layers have been shown to be 
stable for all amplitudes by Matsumura and Nishihara using energy estimates.
Besides the parameter of amplitude appearing in the shock case, 
the boundary-layer case features an additional parameter
measuring displacement of the background profile, which greatly
complicates the resulting case structure.
Moreover, inflow boundary layers turn out to have quite delicate
stability in both large-displacement and large-amplitude limits, 
necessitating the additional use of a mod-two stability index 
studied earlier by Serre and Zumbrun in order to decide stability.
\end{abstract}

\maketitle

\tableofcontents
\bigbreak
\section{Introduction} \label{int}

Consider the isentropic compressible Navier-Stokes equations
\begin{equation}
\begin{split}
\label{eulerian}
\rho_{t}+ (\rho u)_{x} &=0, \\
(\rho u)_{t}+ (\rho u^2)_x + p(\rho)_{x} &= u_{xx}
\end{split}
\end{equation}
on the quarter-plane $x,t \ge 0$,
where $\rho>0$, $u$, $p$ denote density, velocity, and pressure
at spatial location $x$ and time $t$,
with $\gamma$-law pressure function
\begin{equation}\label{gaslaw}
p(\rho) = a_0 \rho^{\gamma},
\qquad
a_0>0, \, \gamma \geq 1,
\end{equation}
and noncharacteristic constant ``inflow'' or ``outflow''
boundary conditions
\begin{equation}\label{inflow}
(\rho, u)(0,t) \equiv (\rho_0, u_0), \qquad u_0>0
\end{equation}
or
\begin{equation}\label{outflow}
u(0,t) \equiv u_0 \qquad u_0<0
\end{equation}
as discussed in \cite{SZ,GMWZ.5,GMWZ.6}.
The sign of the velocity at $x=0$ determines whether
characteristics of the hyperbolic transport equation
$\rho_t+ u\rho_x= f$ enter the domain
(considering $f:=\rho u_x$ as a lower-order forcing term),
and thus whether $\rho(0,t)$ should be prescribed.
The variable-coefficient parabolic equation
$\rho u_t - u_{xx}= g$ requires prescription of $u(0,t)$ in
either case, with $g:= -\rho (u^2/2)_x -p(\rho)_x $.

By comparison, the purely hyperbolic isentropic Euler equations
\begin{equation}
\begin{split}
\label{euler}
\rho_{t}+ (\rho u)_{x} &=0, \\
(\rho u)_{t}+ (\rho u^2)_x + p(\rho)_{x} &= 0
\end{split}
\end{equation}
have characteristic speeds $a= u\pm \sqrt{p'(\rho)}$, hence,
depending on the values of $(\rho, u)(0,t)$, may have one, two,
or no characteristics entering the domain, hence require one,
two, or no prescribed boundary values.
In particular, there is a discrepancy between the number of
prescribed boundary values for \eqref{eulerian} and \eqref{euler}
in the case of mild inflow $u_0>0$ small (two for \eqref{eulerian}, one
for \eqref{euler}) or
strong outflow $u_0<0$ large (one for \eqref{eulerian}, none
for \eqref{euler}), indicating the possibility of {\it boundary layers},
or asymptotically-constant stationary solutions of \eqref{eulerian}:
\begin{equation}\label{BL}
(\rho, u)(x,t)\equiv (\hat \rho, \hat u)(x),
\qquad
\lim_{z\to +\infty} (\hat \rho, \hat u)(z)= (\rho_+, u_+).
\end{equation}
Indeed, existence of such solutions is straightforward to verify
by direct computations on the (scalar) stationary-wave ODE;
see \cite{MeZ,SZ,MN.2,KNZ,GMWZ.5,GMWZ.6} or Section \ref{profsec}.
These may be either of ``expansive'' type, resembling rarefaction
wave solutions on the whole line,  or ``compressive'' type,
resembling viscous shock solutions.

A fundamental question is whether or not such boundary layer solutions
are {\it stable} in the sense of PDE.
For the expansive inflow case, it has been shown
in \cite{MN.2} that {\it all} boundary layers are stable,
independent of amplitude, by energy estimates similar
to those used to prove the corresponding result for rarefactions
on the whole line.
Here, we concentrate on the complementary, {\it compressive case}
(though see discussion, Section \ref{discussion}).

Linearized and nonlinear stability of general
(expansive or compressive) {\it small-amplitude} noncharacteristic
boundary layers of \eqref{eulerian} have been established
in \cite{MN.2,R,KNZ,GMWZ.5}.
More generally, it has been shown in \cite{GMWZ.5,YZ} that linearized
and nonlinear stability are equivalent to {spectral stability},
or nonexistence of nonstable (nonnegative real part) eigenvalues of
the linearized operator about the layer, for boundary layers of arbitrary amplitude.
However, up to now the spectral stability of {\it large-amplitude
compressive} boundary layers has remained largely undetermined.\footnote{
See, however, the
investigations of \cite{SZ} on stability index, or parity of the number
of nonstable eigenvalues of the linearized operator about the layer.}

We resolve this question in the present paper, carrying out a systematic,
global study classifying the stability of all possible compressive
boundary-layer solutions of \eqref{eulerian}.
Our method of analysis is by a combination of asymptotic ODE techniques
and numerical Evans function computations,
following a basic approach introduced recently in \cite{HLZ,BHRZ} for
the study of the closely related shock wave case.
Here, there are interesting complications associated with the
richer class of boundary-layer solutions as compared to possible
shock solutions, the delicate stability properties of the
inflow case, and, in the outflow case, the nonstandard eigenvalue problem
arising from reduction to Lagrangian coordinates.

Our conclusions are, for both inflow and outflow conditions,
that compressive boundary layers that are uniformly noncharacteristic
in a sense to be made precise later (specifically, $v_+$ bounded
away from $1$, in the terminology of Section \ref{profsec})
are {\it unconditionally stable}, independent of amplitude,
on the range $\gamma \in [1,3]$ considered in our numerical
computations.
We show by energy estimates that {\it outflow boundary layers are stable
also in the characteristic limit}.
The omitted characteristic limit in the inflow case, 
analogous to the small-amplitude limit
for the shock case should be treatable by the singular perturbation
methods used in \cite{PZ,FS} to treat the small-amplitude shock case;
however, we do not consider this case here.

In the inflow case, our results, together with
those of \cite{MN.2}, completely resolve the question of stability
of isentropic (expansive or compressive) uniformly noncharacteristic
boundary layers for $\gamma\in [1,3]$, yielding
{\it unconditional stability independent of amplitude or type.}
In the outflow case, we show stability of all {compressive} boundary
layers without the assumption of uniform noncharacteristicity.

\subsection{Discussion and open problems}\label{discussion}
The small-amplitude results obtained in \cite{MN.2,KNZ,R,GMWZ.5}
are of ``general type'', making little use of the specific
structure of the equations.
Essentially, they all require that the difference between the
boundary layer solution and its constant limit at $|x|=\infty$
be small in $L^1$.\footnote{Alternatively, as in
\cite{MN.2,R}, the essentially equivalent condition
that $x\hat v'(x)$ be small in $L^1$.
(For monotone profiles,
$\int_0^{+\infty} |\hat v-v_+|dx=
\pm\int_0^{+\infty} (\hat v-v_+) dx=
\mp\int_0^{+\infty} x\hat v'dx$.)}
As pointed out in \cite{GMWZ.5}, this is
the ``gap lemma'' regime in which standard asymptotic ODE
estimates show that behavior is essentially governed by
the limiting constant-coefficient equations at infinity,
and thus stability may be concluded immediately from stability
(computable by exact solution) of the constant layer
identically equal to the limiting state.
These methods do not suffice to treat either the (small-amplitude)
characteristic limit or the large-amplitude case, which require
more refined analyses.
In particular, up to now, {\it there was no analysis considering boundary
layers approaching a full viscous shock profile},
{\it not even a profile of vanishingly small amplitude}.
Our analysis of this limit indicates why: the appearance of a small
eigenvalue near zero prevents uniform estimates such as would be obtained
by usual types of energy estimates.

By contrast, the large-amplitude results obtained here and (for expansive
layers) in \cite{MN.2}
make use of the specific form of the equations.
In particular, both analyses make use of the advantageous structure
in Lagrangian coordinates.
The possibility to work in Lagrangian coordinates was first pointed out
by Matsumura--Nishihara \cite{MN.2} in the inflow case,
for which the stationary boundary transforms to a moving boundary
with constant speed.
Here we show how to convert the outflow problem also to Lagrangian
coordinates, by converting the resulting variable-speed
boundary problem to a constant-speed one with modified boundary condition.
This trick seems of general use.  
In particular, 
it might be possible that the
energy methods of \cite{MN.2} applied in this framework would yield
unconditional stability of expansive boundary-layers,
completing the analysis of the outflow case.
Alternatively, this case could be attacked by the methods of the
present paper.
These are two further interesting direction for future investigation.

In the outflow case, a further transformation to the
``balanced flux form'' introduced in \cite{PZ}, in which the equations
take the form of the integrated shock equations, 
allows us to establish stability in the characteristic limit by
energy estimates like those of \cite{MN} in the shock case.
The treatment of the characteristic inflow limit by the methods of \cite{PZ,FS}
seems to be another extremely interesting direction for future study.

Finally, we point to the extension of the present methods to
full (nonisentropic) gas dynamics and multidimensions as
the two outstanding open problems in this area.

New features of the present analysis as compared to the shock
case considered in \cite{BHRZ,HLZ} are the presence of two parameters,
strength and displacement, indexing possible boundary layers,
vs. the single parameter of strength in the shock case, and
the fact that the limiting equations in several asymptotic regimes
possess zero eigenvalues, making the limiting stability analysis much
more delicate than in the shock case.
The latter is seen, for example, in the limit as a compressive
boundary layer approaches a full stationary shock solution,
which we show to be spectrally equivalent to the situation
of unintegrated shock equations on the whole line.
As the equations on the line
possess always a translational eigenvalue at $\lambda=0$, we
may conclude existence of a zero at $\lambda=0$ for the limiting equations
and thus a zero {\it near} $\lambda=0$ as we approach this limit,
which could be stable or unstable.
Similarly, the Evans function in the inflow case is shown to
converge in the large-strength limit to a function with a
zero at $\lambda=0$, with the same conclusions;
see Section \ref{description} for further details.

To deal with this latter circumstance, we find it necessary
to make use also of topological information provided by the
stability index of \cite{PW,GZ,SZ}, a mod-two index counting
the parity of the number of unstable eigenvalues.
Together with the information that there is at most one
unstable zero, the parity information provided by the stability
index is sufficient to determine whether an unstable zero
does or does not occur.
Remarkably, in the isentropic case we are able to compute {explicitly}
the stability index for all parameter values, recovering
results obtained by indirect argument in \cite{SZ},
and thereby completing the stability analysis in the presence
of a single possibly unstable zero.


\section{Preliminaries}\label{prelim}
We begin by carrying out a number of preliminary steps similar to
those carried out in \cite{BHRZ,HLZ} for the shock case, but
complicated somewhat by the need to treat the boundary and
its different conditions in the inflow and outflow case.

\subsection{Lagrangian formulation.}\label{lagrangiansec}

The analyses of \cite{HLZ,BHRZ} in the shock wave case were carried
out in Lagrangian coordinates, which proved to be particularly
convenient.
Our first step, therefore, is to convert the Eulerian formulation
\eqref{eulerian} into Lagrangian coordinates similar to those of the shock case.
However, standard Lagrangian coordinates in which the spatial variable
$\tilde x$ is constant on particle paths are not appropriate for the
boundary-value problem with inflow/outflow.
We therefore introduce instead ``psuedo-Lagrangian'' coordinates
\begin{equation}\label{psuedo}
\tilde x:=\int_0^x \rho(y,t)\, dy, \quad \tilde t:=t,
\end{equation}
in which the physical boundary $x=0$ remains fixed at $\tilde x=0$.

Straightforward calculation reveals that in these coordinates
\eqref{eulerian} becomes

\begin{equation}\label{NS}
\begin{aligned}
v_t - s v_{\tilde x} - u_{\tilde x} &= \sigma(t)v_{\tilde x} \\
u_t - s u_{\tilde x} + p(v)_{\tilde x} -\left(\frac{u_{\tilde
x}}{v}\right)_{\tilde x} &= \sigma(t) u_{\tilde x}
\end{aligned}
\end{equation}
on $x>0$, where
\begin{equation}\label{s(t)}
s = -\frac{u_0}{v_0}, \;  \sigma(t) = m(t) - s, \; m(t) := -
\rho(0,t) u(0,t)=-u(0,t)/v(0,t),
\end{equation} so that $m(t)$ is the negative of the momentum at the
boundary $x= \tilde x = 0$. From now on, we drop the tilde, denoting
$\tilde x$ simply as $x$.

\subsubsection{Inflow case}\label{lagin}
For the inflow case, $u_0>0$ so we may prescribe {\em two} boundary
conditions on \eqref{NS}, namely
\begin{equation}
v|_{x=0} = v_0 > 0, \hspace{0.5cm} u|_{x=0} = u_0 > 0
\end{equation} where both $u_0, v_0$ are constant.

\subsubsection{Outflow case}\label{lagout}
For the outflow case, $u_0<0$ so we may prescribe {\em only one}
boundary condition on \eqref{NS}, namely
\begin{equation}
u|_{x=0} = u_0 < 0.
\end{equation} Thus $v(0,t)$ is an unknown in the problem, which
makes the analysis of the outflow case more subtle than that of
the inflow case.

\subsection{Rescaled coordinates}\label{renormalization}
Our next step is to rescale the equations in such a way
that coefficients remain bounded in the strong boundary-layer
limit.
Consider the change of variables
\begin{equation}\label{scaling}
(x,t,v,u) \rightarrow (-\varepsilon s x, \varepsilon s^2 t,
v/\varepsilon, -u/(\varepsilon s)),
\end{equation}
where $\varepsilon$ is chosen so that
\begin{equation}
0 < v_+ < v_- = 1,
\end{equation}
where $v_+$ is the limit as $x\to +\infty$ of the boundary layer
(stationary solution) $(\hat v, \hat u)$ under consideration
and $v_-$ is the limit as $x\to -\infty$ of its continuation
into $x<0$ as a solution of the standing-wave ODE (discussed in
more detail just below).
Under the rescaling \eqref{scaling}, \eqref{NS} becomes
\begin{equation} \label{rescaled}
\begin{aligned}
v_t + v_x - u_x &= \sigma(t)v_x,\\
u_t + u_x + (a v^{-\gamma})_x &= \sigma(t) u_x +
\left(\frac{u_x}{v}\right)_x
\end{aligned}
\end{equation}
where $a = a_0 \varepsilon^{-\gamma-1} s^{-2}$, $\sigma=-u(0,t)/v(0,t)+1$,
on respective domains
$$
x>0 \, \hbox{\rm (inflow case)}
\qquad x<0 \, \hbox{\rm (outflow case)}.
$$

\subsection{Stationary boundary layers}\label{profsec}
Stationary boundary layers 
$$
(v,u)(x,t) = (\bv,\bu)(x)
$$
of \eqref{rescaled} satisfy
\begin{equation}\label{stationarybl}
\begin{split}
(a) \hspace{0.5cm} & \bv' - \bu' = 0 \\
(b) \hspace{0.5cm} & \bu' + (a\bv^{-\gamma}) = \left(\frac{\bu'}{\bv}\right)' \\
(c) \hspace{0.5cm} & (\bv,\bu)|_{x=0} = (v_0,u_0)\\
(d) \hspace{0.5cm} & \lim_{x\rightarrow \pm\infty}(\bv,\bu) =
(v,u)_\pm,
\end{split}
\end{equation}
where (d) is imposed at $+\infty$ in the inflow case, $-\infty$
in the outflow case and (imposing $\sigma=0$) $u_0=v_0$.
Using \eqref{stationarybl}(a) we can reduce
this to the study of the scalar ODE,
\begin{equation}\label{scalarode}
\bv' + (a\bv^{-\gamma})' = \left(\frac{\bv'}{\bv}\right)' \\
\end{equation} with the same boundary conditions at $x=0$ and 
$x=\pm\infty$ as above.  
Taking the antiderivative of this equation
yields
\begin{equation}\label{profeqC}
\bv' = \mathcal{H}_C(\bv) = \bv(\bv + a\bv^{-\gamma} +C ),
\end{equation}
where $C$ is a constant of integration.

Noting that $\mathcal{H}_C$ is convex, we find that there
are precisely two rest points of \eqref{profeqC} whenever boundary-layer
profiles exist, except at the single parameter value on the boundary
between existence and nonexistence of solutions, for which
there is a degenerate rest point (double root of $\mathcal{H}_C$).
Ignoring this degenerate case, we see that boundary layers terminating
at rest point $v_+$ as $x\to +\infty$ must either continue
backward into $x<0$ to terminate at a second rest point $v_-$ as
$x\to -\infty$, or else blow up to infinity as $x\to -\infty$.
The first case we shall call {\it compressive}, the second
{\it expansive}.

In the first case, the extended solution on the whole line
may be recognized as a standing viscous shock wave;
that is, {\it for isentropic gas
dynamics, compressive boundary layers are just restrictions to the half-line
$x\ge 0$ [resp. $x\le 0$] of standing shock waves.}
In the second case, as discussed in \cite{MN.2}, the boundary layers
are somewhat analogous to rarefaction waves on the whole line.
{From here on, we concentrate exclusively on the compressive case.}

With the choice $v_-=1$, we may carry out the integration of
\eqref{scalarode} once more, this time as a definite
integral from $-\infty$ to $x$, to obtain
\begin{equation}\label{profeq}
\bv' = H(\bv) = \bv(\bv-1 + a(\bv^{-\gamma} - 1)),
\end{equation}
where $a$ is found by letting $x \rightarrow +\infty$, yielding
\begin{equation}
\label{RH} a = -\frac{v_+ - 1}{v_+^{-\gamma} - 1} = v_+^\gamma
\frac{1-v_+}{1-v_+^\gamma} \, ;
\end{equation}
in particular,
$ a\sim v_+^\gamma $ in the large boundary layer limit $v_+\to 0$.
This is exactly the equation for viscous shock profiles
considered in \cite{HLZ}.

\subsection{Eigenvalue equations}\label{eigsec}
Linearizing \eqref{rescaled} about $(\bv,\bu)$, we obtain

\begin{equation} \label{linearized}
\begin{split}
& \tv_t + \tv_x - \tu_x = \frac{\tv(0,t)}{ v_0}\bv '  \\
& \tu_t + \tu_x - \left(\frac{h(\bv)}{\bv^{\gamma + 1}} \tilde v  \right)_x
- \left( \frac{\tu_x}{\bv} \right)_x=
\frac{\tv(0,t)}{v_0}\bu '  \\
& (\tv,\tu)|_{x=0} = (\tv_0(t),0) \\
& \lim_{x\rightarrow +\infty}(\tv,\tu) = (0,0)
\end{split}
\end{equation}
where $v_0=\hat v(0)$,
\begin{equation}\label{f}
h(\bv) = -\bv^{\gamma + 1} + a(\gamma - 1) + (a+1)\bv^\gamma
\end{equation}
and $\tv,\tu$ denote perturbations of $\bv, \bu$.

\subsubsection{Inflow case}\label{ineigen}
In the inflow case, $\tilde u(0,t)=\tilde v(0,t)\equiv 0$, yielding
\begin{equation}\label{eigen1}
\begin{aligned}
\lambda v + v_x - u_x &= 0\\
\lambda u + u_x - \left(\frac{h(\bv)}{\bv^{\gamma + 1}} v  \right)_x
&= \left( \frac{u_x}{\bv} \right)_x\\
\end{aligned}
\end{equation}
on $x>0$, with full Dirichlet conditions $(v,u)|_{x=0}=(0,0)$.

\subsubsection{Outflow case}\label{outeigen}
Letting $\tU := (\tv,\tu)^{T}$, $\hat U:=(\hat v, \hat u)^T$,
and denoting by $\CalL$ the operator
associated to the linearization about boundary-layer $(\hat v,\hat u)$,
\begin{equation}\label{linoperator}
\CalL := \partial_x A(x)   - \partial_x B(x)\partial_x,
\end{equation} where
\begin{equation}
A(x) = \left( \begin{array}{cc} 1 & -1 \\
-h(\bv)/\bv^{\gamma + 1} & 1  \\
\end{array} \right), \hspace{0.5cm} B(x) = \left( \begin{array}{cc} 0 & 0 \\
0 & \bv^{-1} \end{array} \right),
\end{equation}
we have
$\tilde U_t - \CalL \tilde U = \frac{\tv_0(t)}{v_0}\hat U'(x)$,
with associated eigenvalue equation
\begin{equation} \label{vectorevalue}
\lambda \tilde U - \CalL \tilde U = \frac{\tv(0,\lambda)}{v_0}\hat U'(x),
\end{equation}
where $\bU' = (\bv',\bu')$.

To eliminate the nonstandard inhomogeneous term on the righthand side of
\eqref{vectorevalue}, we introduce
a ``good unknown" (c.f. \cite{Al,CJLW,GMWZ.3,HuZ.3})
\begin{equation}
U := \tU - \lambda^{-1} \frac{\tv(0,\lambda)}{v_0} \hat U'(x).
\end{equation}
Since $\CalL \bU ' = 0$ by differentiation of the boundary-layer equation,
the system expressed in the good unknown becomes simply
\begin{equation} \label{lin-stability-prob}
 U_t - \CalL U = 0 \hspace{0.5cm} \mbox{in} \; x<0 ,
\end{equation}
or, equivalently, \eqref{eigen1}
with boundary conditions
\begin{equation}
\begin{aligned}
& U|_{x=0} = \frac{\tv(0,\lambda)}{v_0}(1 - \lambda^{-1}\bv '(0) ,
\, -\lambda^{-1} \bu '(0))^{T} \\
& \lim_{x \rightarrow +\infty} U = 0.
\end{aligned}
\end{equation}
Solving for $u|_{x=0}$ in terms of $v|_{x=0}$ and
recalling that $\hat v'=\hat u'$ by \eqref{profeq},
we obtain finally
\begin{equation}\label{outflowBC}
u|_{x=0} =\alpha(\lambda) v|_{x=0},
\qquad
\alpha(\lambda):= \frac{-\bv'(0)}{\lambda  - \bv '(0)}.
\end{equation}

\begin{remark}\label{outrmk}
Problems \eqref{vectorevalue} and \eqref{lin-stability-prob}--\eqref{eigen1}
are evidently equivalent for all $\lambda\ne 0$,
but are not equivalent for $\lambda=0$ (for which the change of coordinates
to good unknown becomes singular).  For, $U=\hat U'$ by inspection
is a solution of \eqref{lin-stability-prob}, but is not
a solution of \eqref{vectorevalue}.
That is, we have introduced by this transformation a spurious
eigenvalue at $\lambda=0$, which we shall have to account for later.
\end{remark}

\subsection{Preliminary estimates}\label{prelimests}

\begin{proposition} [\cite{BHRZ}] \label{profdecay}
For each $\gamma\ge 1$, $0<v_+\le 1/12<v_0 < 1$,
\eqref{profeq} has a unique
(up to translation) monotone
decreasing solution $\hat v$ decaying to endstates $v_\pm$
with a uniform exponential rate for $v_+$ uniformly bounded away from $v_-=1$.
In particular, for $0<v_+\le 1/12$,
\begin{subequations}
\label{decaybd}
\begin{align}
|\bv(x)-v_+|&\le C e^{-\frac{3(x-\delta)} {4}} \quad x\ge \delta,
\label{decaybd_1}\\
|\bv(x)-v_-|&\le
Ce^{\frac{(x-\delta)}{2}} \quad x\le \delta
\label{decaybd_2}
\end{align}
\end{subequations}
where $\delta$ is defined by $\hat v(\delta)=(v_-+v_+)/2$.
\end{proposition}

\begin{proof}
Existence and monotonicity follow trivially
by the fact that \eqref{profeq} is a scalar
first-order ODE with convex righthand side.
Exponential convergence as $x\to +\infty$ follows by
$H(v, v_+) =
(v-v_+)\Big(v - \Big(\frac{1-v_+}{1-v_+^{\gamma}}\Big)
\Big(\frac{1 - \big(\frac{v_+}{v}\big)^{\gamma}}{1 -
\big(\frac{v_+}{v}\big)}\Big)\Big),
$
whence
$v- \gamma \le \frac{H(v,v_+)}{v-v_+}\le v-(1-v_+)$
by $1\le \frac{1-x^\gamma}{1-x}\le \gamma$
for $0\le x\le 1$.  Exponential convergence as $x\to -\infty$
follows by a similar, but more straightforward calculation,
where, in the ``centered'' coordinate $\tilde x:=x-\delta$,
the constants $C>0$ are uniform with respect to $v_+, v_0$.
See \cite{BHRZ} for details.
\end{proof}

The following estimates are established in Appendices \ref{basicproof}
and \ref{outbasicproof}.

\begin{proposition} \label{hf}
Nonstable eigenvalues $\lambda$ of \eqref{eigen1}, i.e., eigenvalues
with nonnegative real part, are confined for any $0<v_+\le 1$ to the
region
\begin{equation}
\label{hfbounds1} \Lambda:= \{\lambda:\, \R(\lambda) + |\I(\lambda)|
\leq \frac 12\Big(2\sqrt{\gamma}+1\Big)^2\}.
\end{equation}
for the inflow case, and to the region
\begin{equation}
\label{hfbounds2} \Lambda:= \{\lambda:\, \R(\lambda) + |\I
(\lambda)| \le  \max\{\frac{3\sqrt2}{2},3\gamma+ \frac{3}{8}\}
\end{equation}
for the outflow case.
\end{proposition}

\subsection{Evans function formulation}\label{evanssec}
Setting $w:=\frac{u'}{\hat v} +\frac{h(\hat v)}{\hat v^{\gamma+1}}v - u$,
we may express \eqref{eigen1} as a first-order system
\begin{equation}\label{firstorder}
W' = A(x,\lambda) W,
\end{equation}
where
\begin{equation}
\label{evans_ode}
A(x,\lambda) = \begin{pmatrix}
0 & \lambda & \lambda\\
0 & 0 & \lambda\\
\hat v & \hat v & f(\hat v)- \lambda\\ \end{pmatrix},
\quad W = \begin{pmatrix}
w\\u-v\\v\end{pmatrix},\quad \prime = \frac{d}{dx},
\end{equation}
where
\begin{equation}\label{feq}
f(\bv) = \bv- \bv^{-\gamma} h(\bv)
= 2\bv - a(\gamma-1)\bv^{-\gamma} - (a+1),
\end{equation}
with $h$ as in \eqref{f} and $a$ as in \eqref{RH},
or, equivalently,
\begin{equation}\label{feq2}
f(\bv) = 2\bv - (\gamma-1)
\Big(\frac{1-v_+}{1-v_+^\gamma}\Big) \Big(\frac{v_+}{\bv}\Big)^{\gamma}
- \Big(\frac{1-v_+}{1-v_+^\gamma}\Big) v_+^{\gamma} -1.
\end{equation}

\begin{remark}\label{shockrel}
The coefficient matrix $A$ may be recognized as a rescaled version
of the coefficient matrix 
$\mathcal{A}$ appearing in the shock case \cite{BHRZ,HLZ},
with
$$
\mathcal{A}=
\begin{pmatrix}1&0&0\\0&1&0\\0&0&\lambda\end{pmatrix}
A\begin{pmatrix}1&0&0\\0&1&0\\0&0&1/\lambda\end{pmatrix}.
$$
The choice of variables $(w,u-v,v)^T$ may be recognized as the modified flux
form of \cite{PZ}, adapted to the hyperbolic--parabolic case.
\end{remark}

Eigenvalues of \eqref{eigen1} correspond to nontrivial solutions $W$
for which the boundary conditions $W(\pm\infty)=0$ are satisfied.
Because $A(x,\lambda)$ as a function of $\bv$ is asymptotically constant
in $x$, the behavior near $x=\pm \infty$ of solutions of
\eqref{evans_ode} is governed by the limiting constant-coefficient
systems
\begin{equation}
\label{apm}
W' = A_\pm(\lambda) W, \qquad
A_\pm(\lambda):=A(\pm \infty,\lambda),
\end{equation}
from which we readily find on the (nonstable) domain $\Re \lambda \ge 0$,
$\lambda\ne 0$ of
interest that there is a one-dimensional
unstable manifold $W_1^-(x)$ of solutions decaying at $x=-\infty$ and
a two-dimensional stable manifold $W_2^+(x) \wedge W_3^+(x)$ of
solutions decaying at $x=+\infty$, analytic in $\lambda$, with
asymptotic behavior
\begin{equation}\label{asymptotics}
W_j^\pm(x,\lambda) \sim e^{\mu_\pm(\lambda) x} V_j^\pm(\lambda)
\end{equation}
as $x\to \pm \infty$, where $\mu_\pm(\lambda)$ and $V_j^\pm(\lambda)$
are eigenvalues and associated analytically chosen
eigenvectors of the limiting coefficient matrices $A_\pm (\lambda)$.
A standard choice of eigenvectors $V_j^\pm$ \cite{GZ,BrZ,BDG,HSZ},
uniquely specifying $W_j^\pm$ (up to constant factor) is obtained by
Kato's ODE \cite{Kato}, a linear, analytic ODE whose solution
can be alternatively characterized by the property
that there exist corresponding left eigenvectors $\tilde V_j^\pm$ such that
\begin{equation}\label{katoprop}
(\tilde V_j\cdot V_j)^\pm \equiv  {\rm constant}, \quad
(\tilde V_j \cdot \dot V_j)^\pm \equiv 0,
\end{equation}
where ``$\, \, \dot{    }\, \,$'' denotes $d/d\lambda$;
for further discussion, see \cite{Kato,GZ,HSZ}.

\subsubsection{Inflow case}\label{inev}
In the inflow case, $0\le x\le +\infty$, we define
the {\it Evans function} $D$
as the analytic function
\begin{equation}\label{evansdef}
D_{\rm in}(\lambda): = \det(W_1^0, W_2^+, W_3^+)_{\mid x=0},
\end{equation}
where $W^+_j$ are as defined above, and $W^0_1$ is a solution
satisfying the boundary conditions $(v,u)=(0,0)$ at $x=0$,
specifically,
\begin{equation}\label{w01}
W^0_1|_{x=0}=(1,0,0)^T.
\end{equation}
With this definition, eigenvalues of $\mathcal{L}$ 
correspond to zeroes of $D$ both
in location and multiplicity; moreover, the Evans function
extends analytically to $\lambda=0$, i.e., to all of $\Re \lambda \ge 0$.
See \cite{AGJ,GZ,MZ.3,Z.3} for further details.

Equivalently,
following \cite{PW,BHRZ}, we may express the Evans function as
\begin{equation}\label{adjevans}
D_{\rm in}(\lambda)=
\big( W_1^0 \cdot \widetilde{W}_1^+ \big)_{\mid x=0},
\end{equation}
where $\widetilde{W}_1^+(x)$ spans the one-dimensional
unstable manifold of solutions decaying at $x=+\infty$
(necessarily orthogonal to the span of
$W_2^+(x)$ and  $W_3^+(x)$)
of the adjoint eigenvalue ODE
\begin{equation}\label{adjode}
\widetilde{W}' =
- A(x,\lambda)^*  \widetilde{W}.
\end{equation}
The simpler representation \eqref{adjevans}
is the one that we shall use here.

\subsubsection{Outflow case}\label{outev}
In the outflow case, $-\infty \le x\le 0$, we define
the {\it Evans function} as
\begin{equation}\label{evansdef2}
D_{\rm out}(\lambda): = \det(W_1^-, W_2^0, W_3^0)_{\mid x=0},
\end{equation}
where $W^-_1$ is as defined above, and $W^0_j$ are a basis
of solutions of \eqref{firstorder}
satisfying the boundary conditions \eqref{outflowBC},
specifically,
\begin{equation}\label{w023}
W^0_2|_{x=0}=(1,0,0)^T,
\qquad
W^0_3|_{x=0}=\Big(0, -\frac{\lambda}{\lambda - \bv '(0)}, 1\Big)^T,
\end{equation}
or, equivalently, as
\begin{equation}\label{adjevans2}
D_{\rm out}(\lambda)=\big(W_1^-\cdot \widetilde{W}_1^0 \big)_{\mid x=0},
\end{equation}
where
\begin{equation}\label{tildew1}
\widetilde{W}^0_1= \Big(0, -1,
-\frac{\bar \lambda}{\bar \lambda - \bv '(0)}\Big)^T
\end{equation}
is the solution of the adjoint eigenvalue ODE dual to $W^0_2$ and $W^0_3$.

\begin{remark}\label{outevansrmk}
As discussed in Remark \ref{outrmk},
$D_{\rm out}$ has a spurious zero at $\lambda=0$
introduced by the coordinate change to ``good unknown''.
\end{remark}

\section{Main results}\label{description}

We can now state precisely our main results.

\subsection{The strong layer limit}\label{strong}
Taking a formal limit as $v_+\to 0$ of the rescaled equations
\eqref{rescaled} and recalling that $a\sim v_+^\gamma$, we obtain
a {limiting evolution equation}
\begin{equation}
\begin{split}
\label{pressureless}
v_t + v_x - u_x &= 0,\\
u_t + u_x  &= \left(\frac{u_x}{v}\right)_x
\end{split}
\end{equation}
corresponding to a {\it pressureless gas}, or $\gamma=0$.

The associated {limiting profile equation} $v' =v(v-1)$
has explicit solution
\begin{equation}\label{v^0}
\hat v^0(x)= \frac{1-\tanh \big(\frac{x-\delta}{2}\big)}{2},
\end{equation}
$\hat v^0(0)= \frac{1-\tanh (-\delta/2)}{2}= v_0$;
the limiting eigenvalue system is
$
W' = A^0(x,\lambda) W,
$
\begin{equation}
\label{limevans_ode}
A^0(x,\lambda) = \begin{pmatrix}0 & \lambda & \lambda\\0 & 0 & \lambda\\
\bv^0& \bv^0 &f^0(\bv^0)-\lambda \end{pmatrix},
\end{equation}
where
$f^0(\bv^0) = 2\bv^0 - 1 = -\tanh\big(\frac{x+\delta}{2}\big).$

Convergence of the profile and eigenvalue equations is {\it uniform} on any
interval $\hat v^0\ge \epsilon >0$, or, equivalently,
$x-\delta \le L$, for $L$ any positive constant, where
the sequence of coefficient matrices is
therefore a {\it regular perturbation} of its limit.
Following \cite{HLZ}, we call $x\le L+\delta$ the ``regular region''.
For $\hat v_0\to 0$ on the other hand, or $x\to \infty$, the limit
is less well-behaved, as may be seen by the fact that $\partial f/\partial \hat v\sim \hat v^{-1}$ as $\hat v\to v_+$, a consequence of the
appearance of $\big(\frac{v_+}{\hat v}\big)$ in
the expression \eqref{feq2} for $f$.
Similarly, $A(x,\lambda)$ does not converge
to $A_+(\lambda)$ as $x\to +\infty$ with uniform exponential rate
independent of $v_+$, $\gamma$, but rather as $C\hat v^{-1}e^{-x/2}$.
As in the shock case, this makes problematic the treatment of
$x\ge L+\delta$.
Following \cite{HLZ} we call
$x\ge L+\delta$ the ``singular region''.

To put things in another way, the effects of pressure are not lost
as $v_+\to 0$, but rather pushed to $x=+\infty$, where they must
be studied by a careful boundary-layer analysis.
(Note: this is not a boundary-layer in the same sense
as the background solution, nor is it a
singular perturbation in the usual sense,
at least as we have framed the problem here.)

\begin{remark}\label{newparam}
A significant difference from the shock case of
\cite{HLZ} is the appearance of the second parameter $v_0$
that survives in the $v_+\to 0$ limit.
\end{remark}

\subsubsection{Inflow case}
Observe that the limiting coefficient matrix
\begin{equation}\label{A0+}
\begin{aligned}
A^0_+(\lambda) &:=A^0(+\infty, \lambda)
=
\begin{pmatrix}0 & \lambda & \lambda\\0 & 0 & \lambda\\
0& 0 &-1 -\lambda \end{pmatrix},
\end{aligned}
\end{equation}
is nonhyperbolic (in ODE sense) for all $\lambda$,
having eigenvalues $0,0,-1-\lambda$; in particular,
the stable manifold drops to dimension one in the limit $v_+\to 0$,
and so the prescription of an associated Evans function is
{\it underdetermined}.

This difficulty is resolved by a careful boundary-layer analysis
in \cite{HLZ}, determining a special ``slow stable'' mode
$$
 V_2^+ \pm(1,0,0)^T
$$
augmenting the ``fast stable'' mode
$$
V_3:=(\lambda/\mu)(\lambda/\mu+1), \lambda/\mu, 1)^T
$$
associated with the single stable eigenvalue
$\mu=-1-\lambda$ of $A^0_+$.
This determines a {\it limiting Evans function} $D^0_{\rm in}(\lambda)$
by the prescription \eqref{evansdef}, \eqref{asymptotics} of
Section \ref{evanssec},
or alternatively via \eqref{adjevans} as
\begin{equation}\label{duallimD}
D^0_{\rm in}(\lambda)=
\big(W_1^{00}\cdot \widetilde{W}_1^{0+} \big)_{\mid x=0},
\end{equation}
with $\widetilde{W}_1^{0+}$ defined analogously as a solution of
the adjoint limiting system
lying asymptotically at $x=+\infty$ in direction
\begin{equation}\label{tildeV}
\widetilde V_1:=
(0, -1, \bar \lambda/\bar \mu)^T
\end{equation}
orthogonal to the span of $V_2$ and $V_3$,
where `` $\bar{ }$ '' denotes complex conjugate,
and $W^{00}_1$ defined as the solution of the limiting eigenvalue
equations satisfying boundary condition \eqref{w01},
i.e., $(W_1^{00})_{\mid x=0}= (1,0,0)^T$.

\subsubsection{Outflow case}
We have no such difficulties in the outflow case, since
$A^0_-=A^0(-\infty)$ remains uniformly hyperbolic,
and we may define a limiting Evans function $D^0_{\rm out}$ directly by
\eqref{evansdef2}, \eqref{asymptotics}, \eqref{tildew1},
at least so long as $v_0$ remains bounded from zero.
(As perhaps already hinted by Remark \ref{newparam}, there are complications
associated with the double limit $(v_0,v_+)\to (0,0)$.)

\subsection{Analytical results}\label{analytical}
With the above definitions, we have the following main theorems
characterizing the strong-layer limit $v_+\to 0$ as well as the
limits $v_0\to 0, \, 1$.

\begin{theorem}\label{mainthm}
For $v_0\ge \eta>0$ and $\lambda$ in any compact subset
of $\Re \lambda \ge 0$,
$D_{\rm in}(\lambda)$ and $D_{\rm out}(\lambda)$ converge
uniformly to $D^0_{\rm in}(\lambda)$ and $D^0_{\rm out}(\lambda)$
as $v_+\to 0$.
\end{theorem}


\begin{theorem}\label{main3}
For $\lambda$ in any compact subset of $\Re \lambda \ge 0$
and $v_+$ bounded from $1$,
$D_{\rm in}(\lambda)$, appropriately renormalized by a nonvanishing
analytic factor,
converges uniformly as $v_0\to 1$ to the Evans function for the (unintegrated)
eigenvalue equations of the associated  viscous shock wave connecting
$v_-=1$ to $v_+$;
likewise, $D^0_{\rm out}(\lambda)$, appropriately renormalized,
converges uniformly as $v_0\to 0$ to the same limit for $\lambda$ uniformly
bounded away from zero.
\end{theorem}

By similar computations, we obtain also the following direct result.

\begin{theorem}\label{main2}
Inflow boundary layers are stable for $v_0$ sufficiently small.
\end{theorem}

We have also the following parity information, obtained by
stability-index computations as in \cite{SZ}.\footnote{Indeed,
these may be deduced from the results of \cite{SZ}, taking account
of the difference between Eulerian and Lagrangian coordinates.}

\begin{lemma}[Stability index]\label{index}
For any $\gamma\ge 1$, $v_0$, and $v_+$, $D_{\rm in}(0)\ne 0$,
hence the number of unstable roots of $D_{\rm in}$ is even;
on the other hand $D^0_{\rm in}(0)=0$
and $\lim_{v_0\to 0}D^0_{\rm in}(\lambda)\equiv 0$.
Likewise, $(D^0_{\rm in})'(0)$, $D_{\rm out}'(0)\ne
0$, $(D^0_{\rm out})'(0)\ne 0$, hence the number of nonzero unstable
roots of $D^0_{\rm in}$, $D_{\rm out}$, $D^0_{\rm out}$ is even.
\end{lemma}

Finally, we have the following auxiliary results established 
by energy estimates in Appendices \ref{stronglimit}, \ref{outnonv},
\ref{char}, and \ref{nonvanish-expansive-inflow}.

\begin{proposition}\label{redenergy}
The limiting Evans function $D^0_{\rm in}$ is nonzero for
$\lambda \ne 0$ on $\R \lambda \ge 0$, for all $1>v_0>0$.
The limiting Evans function $D^0_{\rm out}$ is nonzero for
$\lambda \ne 0$ on $\R \lambda \ge 0$, for $1>v_0> v_*$,
where $v_* \approx 0.0899$ is determined by the functional equation
$v_*= e^{-2/(1-v_*)^2}$.
\end{proposition}

\begin{proposition} \label{charsmallamp}
Compressive outflow boundary layers are stable for $v_+$
sufficiently close to $1$.
\end{proposition}

\begin{proposition}[\cite{MN.2}] \label{expansive}
Expansive inflow boundary layers are stable for all 
parameter values.
\end{proposition}

Collecting information, we have the following analytical stability results.

\begin{corollary}\label{v01}
For $v_0$ or $v_+$ sufficiently small, compressive
inflow boundary layers are stable.
For $v_0$ sufficiently small, 
$v_+$ sufficiently close to $1$,
or $v_0> v_*\approx .0899$ and
$v_+$ sufficiently small, compressive
outflow layers are stable.
Expansive inflow boundary layers are stable
for all parameter values.
\end{corollary}

Stability of inflow boundary layers 
in the characteristic limit $v_+\to 1$ is not treated here,
but should be treatable analytically by the asymptotic ODE methods
used in \cite{PZ,FS} to study the small-amplitude (characteristic)
shock limit.
This would be an interesting direction for future investigation.
The characteristic limit is not accessible numerically, since the
exponential decay rate of the background profile
decays to zero as $|1-v_+|$, so that the numerical
domain of integration needed to resolve the eigenvalue
ODE becomes infinitely large as $v_+\to 1$.

\begin{remark}\label{weakreg}
Stability in the noncharacteristic weak layer limit $v_0\to v_+$ [resp. $1$]
in the inflow [outflow] case, for $v_+$ bounded away from the
strong and characteristic limits $0$ and $1$ has already
been established in \cite{GMWZ.5,R}.
Indeed, it is shown in \cite{GMWZ.5} that the Evans function
converges to that for a constant solution, and this is a {\it regular}
perturbation.
\end{remark}

\begin{remark}\label{extra}
Stability of $D^0_{\rm in}$, $D^0_{\rm out}$ may also be
determined numerically, in particular in the
region $v_0\le v_*$ not covered by Proposition \ref{redenergy}.
\end{remark}

\subsection{Numerical results}\label{numresults}
The asymptotic results of Section \ref{analytical} reduce
the problem of
(uniformly noncharacteristic, $v_+$ bounded away from $v_-=1$)
boundary layer stability to a bounded parameter
range on which the Evans function may be efficiently computed
numerically in a way that is uniformly well-conditioned; see \cite{BrZ}.
Specifically, we may map a semicircle 
$$
\partial\{\Re \lambda \ge 0\}\cap \{|\lambda|\le 10\}
$$
enclosing $\Lambda$ for $\gamma\in [1,3]$ by
$D^0_{\rm in}$, $D^0_{\rm out}$, $D_{\rm in}$, $D_{\rm out}$ and
compute the winding number of its image about the origin
to determine the number of zeroes of the various Evans functions
within the semicircle, and thus within $\Lambda$.
For details of the numerical algorithm, see \cite{BHRZ,BrZ}.

In all cases, we obtain results consistent with {stability};
that is, a winding number of zero or one, depending on the situation.
In the case of a single nonzero root, we know from our limiting
analysis that this root may be quite near $\lambda=0$, making
delicate the direct determination of its stability; however,
in this case we do not attempt to determine the stability numerically,
but rely on the analytically computed stability index to conclude
stability.
See Section \ref{computations} for further details.


\subsection{Conclusions}\label{conclusions}
As in the shock case \cite{BHRZ,HLZ}, our results indicate
{\it unconditional stability} of uniformly noncharacteristic
boundary-layers for isentropic
Navier--Stokes equations (and, for outflow layer, in
the characteristic limit as well), 
despite the additional complexity of the boundary-layer case.
However, two additional comments are in order, perhaps related.
First, we point out that the apparent symmetry of Theorem
\ref{main3} in the $v_0\to 0$ outflow and $v_0\to 1$ inflow limits
is somewhat misleading.
For, the limiting, shock Evans function possesses a single zero
at $\lambda=0$, indicating that stability of inflow boundary
layers is somewhat delicate as $v_0\to 1$: specifically, they
have an eigenvalue near zero, which, though stable, is (since
vanishingly small in the shock limit) not ``very'' stable.
Likewise, the limiting Evans function $D^0_{\rm in}$ as $v_+\to 0$
possesses a zero at $\lambda=0$, with the same conclusions.

By contrast, the Evans functions of outflow boundary layers
possess a spurious zero at $\lambda=0$, so that convergence to
the shock or strong-layer limit in this case implies the {\it absence}
of any eigenvalues near zero, or ``uniform'' stability as $v_+\to 0$.
In this sense, strong outflow boundary layers appear to be more stable
than inflow boundary layers.
One may make interesting comparisons to physical attempts
to stabilize laminar flow along an air- or hydro-foil by suction (outflow)
along the boundary.
See, for example, the interesting treatise \cite{S}.

Second, we point out the result of {\it instability} obtained
in \cite{SZ} for inflow boundary-layers of the full (nonisentropic)
ideal-gas equations for appropriate ratio of the coefficients
of viscosity and heat conduction.
This suggests that the small eigenvalues of the strong inflow-layer
limit may in some cases perturb to the unstable side.
It would be very interesting to make these connections more precise,
as we hope to do in future work.

\section{Boundary-layer analysis}\label{singular}

Since the structure of \eqref{evans_ode} is essentially the same
as that of the shock case, we may follow exactly the treatment
in \cite{HLZ} analyzing the flow of \eqref{evans_ode}
in the singular region $x\to +\infty$.
As we shall need the details for further computations
(specifically, the proof of Theorem \ref{main2}),
we repeat the analysis here in full.

Our starting point is the observation that
\begin{equation}
\label{a+}
A(x,\lambda) = \begin{pmatrix}0 & \lambda & \lambda\\0 & 0 & \lambda\\
\hat v& \hat v &f(\hat v)-\lambda \end{pmatrix}
\end{equation}
is approximately block upper-triangular for $\hat v$ sufficiently small,
with diagonal blocks
$\begin{pmatrix}
0 & \lambda \\
0 &  0\\
\end{pmatrix}$
and
$\begin{pmatrix}
f(\hat v)-\lambda
\end{pmatrix}$
that are uniformly spectrally separated on $\R \lambda \ge 0$,
as follows by
\begin{equation}\label{fneg}
f(\hat v)\le \hat v-1 \le -3/4.
\end{equation}
We exploit this structure by a judicious coordinate change
converting \eqref{evans_ode}
to a system in exact upper triangular form, for which the
decoupled ``slow'' upper lefthand $2\times 2$ block undergoes
a {\it regular perturbation} that can be analyzed by standard
tools introduced in \cite{PZ}.
Meanwhile, the fast, lower righthand $1\times 1$ block, since
scalar, may be solved exactly.

\subsection{Preliminary transformation}\label{pretrans}
We first block upper-triangularize by a static (constant) coordinate
transformation the limiting matrix
\begin{equation}
\label{lima+}
A_+=A(+\infty,\lambda) = \begin{pmatrix}0 & \lambda & \lambda\\0 & 0 & \lambda\\
v_+& v_+ &f(v_+)-\lambda \end{pmatrix}
\end{equation}
at $x=+\infty$ using special block lower-triangular transformations
\begin{equation}\label{statictrans}
R_+:=\begin{pmatrix}
I& 0\\
v_+ \theta_+ & 1\\
\end{pmatrix},
\qquad
L_+:=R_+^{-1}=\begin{pmatrix}
I& 0\\
-v_+\theta_+ & 1\\
\end{pmatrix},
\end{equation}
where $I$ denotes the $2\times 2$ identity matrix and
$\theta_+\in \C^{1\times 2}$ is a $1\times 2$ row vector.

\begin{lemma}\label{pretranslem}
On any compact subset of $\R \lambda \ge 0$, for each $v_+>0$
sufficiently small,
there exists a unique $\theta_+=\theta_+(v_+,\lambda)$ such that
$\hat A_+:=L_+A_+R_+$ is upper block-triangular,
\begin{equation}\label{hata+}
\begin{aligned}
\hat A_+&=
\begin{pmatrix}
\lambda (J+ v_+\bb1 \theta_+) &  \lambda \bb1 \\
0 &  f(v_+)-\lambda -\lambda v_+ \theta_+\bb1 \\
\end{pmatrix},
\end{aligned}
\end{equation}
where $J=\begin{pmatrix} 0 & 1 \\ 0 & 0 \end{pmatrix}$
and
$\bb1=\begin{pmatrix} 1 \\ 1 \\ \end{pmatrix} $,
satisfying a uniform bound
\begin{equation}\label{theta+bd}
|\theta_+|\le C.
\end{equation}
\end{lemma}

\begin{proof}
Setting the $2-1$ block of $\hat A_+$ to zero, we obtain the
matrix equation
$$
\theta_+ (aI-\lambda J)
=  -\bb1^T + \lambda v_+ \theta_+ \bb1 \theta_+,
$$
where $a=f(v_+)-\lambda$, or, equivalently, the fixed-point equation
\begin{equation}\label{fix1}
\theta_+ =
(aI-\lambda J)^{-1}
\Big( -\bb1^T + \lambda v_+\theta_+ \bb1 \theta_+\Big).
\end{equation}
By $\det (aI-\lambda J)= a^2\ne 0$,
$(aI-\lambda J)^{-1}$ is uniformly bounded
on compact subsets of $\R \lambda \ge 0$
(indeed, it is uniformly bounded on all of $\R \lambda \ge 0$),
whence, for $|\lambda|$ bounded and $v_+$ sufficiently small,
there exists a unique
solution by the Contraction Mapping Theorem,
which, moreover, satisfies \eqref{theta+bd}.
\end{proof}

\subsection{Dynamic triangularization}\label{dynamic}
Defining now $Y:=L_+W$ and
$$
\begin{aligned}
&\hat A(x,\lambda)= L_+A(x, \lambda) R_+(x,\lambda)=\\
&\quad \begin{pmatrix}\\
\lambda (J+v_+\bb1 \theta_+) &\quad &  \lambda \bb1 \\
(\hat v-v_+) \bb1^T
  -v_+(f(\hat v)-f(v_+))\theta_+
& \quad&  f(\hat v)-\lambda -\lambda v_+ \theta_+\bb1 \\
\end{pmatrix},
\end{aligned}
$$
we have converted \eqref{evans_ode} to an asymptotically block
upper-triangular system
\begin{equation}\label{tri1}
Y'=\hat A(x,\lambda) Y,
\end{equation}
with $\hat A_+=\hat A(+\infty, \lambda)$ as in \eqref{hata+}.
Our next step is to choose a {\it dynamic} transformation of
the same form
\begin{equation}\label{dyntrans}
\tilde R:=\begin{pmatrix}
I& 0\\
\tilde \Theta & 1\\
\end{pmatrix},
\qquad
\tilde L:=\tilde R^{-1}=\begin{pmatrix}
I& 0\\
-\tilde \Theta & 1\\
\end{pmatrix},
\end{equation}
converting \eqref{tri1} to an exactly block upper-triangular
system, with $\tilde \Theta$ uniformly exponentially decaying
at $x=+\infty$: that is, a {\it regular perturbation} of the identity.

\begin{lemma}\label{dyntranslem}
On any compact subset of $\R \lambda \ge 0$,
for $L$ sufficiently large and each $v_+>0$ sufficiently small,
there exists a unique $\Theta=\Theta_+(x,\lambda, v_+)$ such that
$\tilde A:=\tilde L \hat A(x,\lambda)\tilde R
+ \tilde L'\tilde R$ is upper block-triangular,
\begin{equation}\label{ahat}
\begin{aligned}
\tilde A&=
\begin{pmatrix}
\lambda (J+v_+\bb1 \theta_+ + \bb1 \tilde \Theta) &  \lambda \bb1 \\
0 &  f(\hat v)-\lambda -\lambda \theta_+\bb1 -\lambda \tilde \Theta \bb1 \\
\end{pmatrix},
\end{aligned}
\end{equation}
and
$\tilde \Theta(L)=0$, satisfying a uniform bound
\begin{equation}\label{Thetabd}
|\tilde \Theta(x,\lambda, v_+)|\le Ce^{-\eta x},
\qquad \eta>0, \, x\ge L,
\end{equation}
independent of the choice of $L$, $v_+$.
\end{lemma}

\begin{proof}
Setting the $2-1$ block of $\tilde A$ to zero and computing
$$
\tilde L'\tilde R=
\begin{pmatrix}
0 & 0\\
-\tilde \Theta' & 0
\end{pmatrix}
\begin{pmatrix}
I & 0\\
\tilde \Theta & I
\end{pmatrix}
=
\begin{pmatrix}
0 & 0\\
-\tilde \Theta' & 0,
\end{pmatrix}
$$
we obtain the matrix equation
\begin{equation}\label{mat}
\tilde \Theta' - \tilde \Theta \big(aI-\lambda (J +v_+\bb1\theta_+)\big)
=  \zeta+ \lambda\tilde \Theta \bb1 \tilde \Theta,
\end{equation}
where $a(x):= f(\hat v)-\lambda -\lambda v_+ \theta_+\bb1 $ 
and the forcing term 
$$
\zeta:=
-(\hat v-v_+) \bb1^T
  +v_+(f(\hat v)-f(v_+))\theta_+
$$
by derivative estimate $df/d\hat v\le C\hat v^{-1}$
together with the Mean Value Theorem
is uniformly exponentially decaying:
\begin{equation}\label{zetabd}
\begin{aligned}
|\zeta|\le C |\hat v-v_+|\le
C_2 e^{-\eta x},
\qquad
\eta>0.
\end{aligned}
\end{equation}

Initializing $\tilde \Theta(L)=0$, we obtain by Duhamel's Principle/Variation
of Constants the representation (supressing the argument $\lambda$)
\begin{equation}\label{duhamel}
\tilde \Theta(x)=
\int_L^{x}
S^{y\to x}
(\zeta+ \lambda\tilde \Theta \bb1 \tilde \Theta)(y)
\, dy,
\end{equation}
where $S^{y\to x}$ is the solution operator for the homogeneous
equation
$$
\tilde \Theta' - \tilde \Theta \big(aI-\lambda (J +v_+\bb1\theta_+)\big)=0,
$$
or, explicitly,
$$
S^{y\to x}=
e^{\int_y^x a(y)dy}
e^{ -\lambda (J +v_+\bb1\theta_+)(x-y)}.
$$

For $|\lambda|$ bounded and $v_+$ sufficiently small,
we have by matrix perturbation theory
that the eigenvalues of $ -\lambda (J +v_+\bb1\theta_+)$
are small and the entries are bounded, hence
$$
|e^{ -\lambda (J +v_+\bb1\theta_+)z}|\le Ce^{\epsilon z}
$$
for $z\ge 0$.  Recalling the uniform spectral gap
$\R a =f(\hat v)-\R \lambda \le -1/2$ for $\R \lambda \ge 0$,
we thus have
\begin{equation}\label{Sbd}
|S^{y\to x}|\le Ce^{\eta (x-y)}
\end{equation}
for some $C$, $\eta>0$.
Combining \eqref{zetabd} and \eqref{Sbd}, we obtain
\begin{equation}
\begin{aligned}
\Big|\int_L^{x} S^{y\to x} \zeta(y)\, dy\Big|&\le
\int_L^x
C_2 e^{-\eta(x-y)}e^{-(\eta/2) y} dy \\
& =C_3 e^{-(\eta/2)x}.
\end{aligned}
\end{equation}

Defining $\tilde \Theta(x) =:\tilde \theta(x) e^{-(\eta/2)x}$
and recalling \eqref{duhamel} we thus have
\begin{equation}\label{duhamel2}
\tilde \theta(x)=
f + e^{(\eta/2)x}\int_L^{x}
S^{y\to x}
 e^{-\eta y}\lambda\tilde\theta \bb1 \tilde \theta(y)
\, dy,
\end{equation}
where $f:= e^{(\eta/2)x}\int_L^{x} S^{y\to x} \zeta(y) \, dy$
is uniformly bounded, $|f|\le C_3$, and
$
e^{(\eta/2)x}\int_L^{x} S^{y\to x}
e^{-\eta y}\lambda\tilde \theta \bb1 \tilde \theta(y)
\, dy
$
is contractive with arbitrarily small contraction constant $\epsilon>0$
in $L^\infty[L,+\infty)$ for $|\tilde \theta|\le 2C_3$ for $L$ sufficiently
large, by the calculation
$$
\begin{aligned}
&\Big| e^{(\eta/2)x}\int_L^{x} S^{y\to x}
e^{-\eta y}\lambda\tilde \theta_1 \bb1 \tilde \theta_1(y)
-
e^{(\eta/2)x}\int_L^{x} S^{y\to x}
e^{-\eta y}\lambda\tilde \theta_2 \bb1 \tilde \theta_2(y)
\Big|
\\
& \qquad \qquad \le
\Big|e^{(\eta/2)x}\int_L^{x} Ce^{-\eta(x-y)} e^{-\eta y} \, dy \Big|
|\lambda|
\|\tilde \theta_1- \tilde \theta_2\|_\infty
\max_j \|\tilde \theta_j\|_\infty
\\
& \qquad \qquad \le
e^{-(\eta/2)L}\Big|\int_L^{x} Ce^{-(\eta/2)(x-y)}  \, dy \Big|
|\lambda|
\|\tilde \theta_1- \tilde \theta_2\|_\infty
\max_j \|\tilde \theta_j\|_\infty \\
& \qquad \qquad =
C_3e^{-(\eta/2)L}
|\lambda|
\|\tilde \theta_1- \tilde \theta_2\|_\infty
\max_j \|\tilde \theta_j\|_\infty.
\end{aligned}
$$
It follows by the Contraction Mapping Principle that there exists
a unique solution $\tilde \theta$ of fixed point equation
\eqref{duhamel2} with $|\tilde \theta(x)|\le 2C_3$
for $x\ge L$,
or, equivalently (redefining the unspecified constant $\eta$), \eqref{Thetabd}.
\end{proof}

\subsection{Fast/Slow dynamics}\label{slow}
Making now the further change of coordinates
$$
Z=\tilde LY
$$
and computing
$$
\begin{aligned}
(\tilde LY)'=\tilde L Y' + \tilde L' Y
&=(\tilde LA_++\tilde L')Y,\\
&=(\tilde LA_+\tilde R+\tilde L'\tilde R)Z,\\
\end{aligned}
$$
we find that we have converted \eqref{tri1} to a block-triangular system
\begin{equation} \label{tri2}
Z'=\tilde AZ=
\begin{pmatrix}
\lambda (J+v_+\bb1 \theta_+ + \bb1 \tilde \Theta) &  \lambda\bb1 \\
0 &  f(\hat v)-\lambda -\lambda v_+ \theta_+\bb1 -\lambda\tilde \Theta \bb1 \\
\end{pmatrix}Z,
\end{equation}
related to the original eigenvalue system \eqref{evans_ode} by
\begin{equation}\label{WZ}
W=LZ,\quad
R:=R_+R=
\begin{pmatrix}
I & 0\\
\Theta & 1
\end{pmatrix},
\quad
L:=R^{-1}=
\begin{pmatrix}
I & 0\\
-\Theta & 1
\end{pmatrix},
\end{equation}
where
\begin{equation}\label{Theta}
\Theta= \tilde \Theta + v_+ \theta_+.
\end{equation}

Since it is triangular, \eqref{tri2} may be solved completely
if we can solve the component systems associated with its diagonal
blocks.  The {\it fast system}
$$
z'=
\Big(f(\hat v)-\lambda -\lambda v_+ \theta_+\bb1 -
\lambda \tilde \Theta \bb1 \Big)z
$$
associated to the lower righthand block features rapidly-varying
coefficients.  However, because it is scalar,
it can be solved explicitly by exponentiation.

The {\it slow system }
\begin{equation} \label{slowsys}
z'= \lambda(J+v_+\bb1 \theta_+ + \bb1 \tilde \Theta ) z
\end{equation}
associated to the upper lefthand block, on the other hand,
by \eqref{Thetabd}, is an exponentially decaying
perturbation of a constant-coefficient system
\begin{equation}\label{cc}
z'= \lambda (J+v_+\bb1 \theta_+)z
\end{equation}
that can be explicitly solved by exponentiation, and thus
can be well-estimated by comparison with \eqref{cc}.
A rigorous version of this statement is given by the
{\it conjugation lemma} of \cite{MeZ}:

\begin{proposition}[\cite{MeZ}] \label{conjugation}
Let $M(x,\lambda)=M_+(\lambda)+ \Theta(x,\lambda)$,
with $M_+$ continuous in $\lambda$ and $|\Theta(x,\lambda)|\le Ce^{-\eta x}$,
for $\lambda$ in some compact set $\Lambda$.
Then, there exists a globally invertible matrix
$P(x,\lambda)=I + Q(x,\lambda)$ such that the
coordinate change $z=Pv$ converts the variable-coefficient ODE
$
z'=M(x,\lambda)z
$
to a constant-coefficient equation
$$
v'=M_+(\lambda)v,
$$
satisfying for any $L$, $0<\hat \eta < \eta$ a uniform bound
\begin{equation}\label{qdecay}
|Q(x,\lambda)|\le
C(L,\hat \eta, \eta, \max |(M_+)_{ij}|, \dim M_+)e^{-\hat \eta x}
\quad \hbox{for $x\ge L$}.
\end{equation}
\end{proposition}

\begin{proof}
See \cite{MeZ,Z.3}, or Appendix C, \cite{HLZ}.
\end{proof}

By Proposition \ref{conjugation}, the solution operator for \eqref{slowsys}
is given by
\begin{equation}\label{slowsoln}
P(y,\lambda) e^{\lambda (J+v_+\bb1 \theta_+(\lambda, v_+))(x-y)}
P(x,\lambda)^{-1},
\end{equation}
where $P$ is a uniformly small perturbation of the identity
for $x\ge L$ and $L>0$ sufficiently large.

\section{Proof of the main theorems}\label{proofsec}

With these preparations, we turn now to the proofs of the main theorems.

\subsection{Boundary estimate}\label{estW1}
We begin by recalling the following estimates established in
\cite{HLZ} on $\widetilde W_1^+(L+\delta)$, that is, the value of
the dual mode $\widetilde W_1^+$ appearing in \eqref{adjevans}
at the boundary $x=L+\delta$ between regular and singular regions.
For completeness, and because we shall need the details in further
computations, we repeat the proof in full.

\begin{lemma}[\cite{HLZ}]\label{matching}
For $\lambda$ on any compact subset of $\R \lambda \ge 0$,
and $L>0$ sufficiently large,
with $\widetilde W_1^+$ normalized as in \cite{GZ,PZ,BHRZ},
\begin{equation}\label{wcon}
|\widetilde W_1^+(L+\delta)-\widetilde V_1| \le Ce^{-\eta L}
\end{equation}
as $v_+\to 0$, uniformly in $\lambda$, where $C$, $\eta>0$ are
independent of $L$
and
$$
\widetilde V_1:= (0, -1, \lambda/\mu)^T
$$
is the limiting direction vector \eqref{tildeV}
appearing in the definition of $D^0_{\rm in}$.
\end{lemma}

\begin{corollary}[\cite{HLZ}]\label{matching2}
Under the hypotheses of Lemma \ref{matching},
\begin{equation}\label{limcon}
|\tilde W_1^{0+}(L+\delta)-\widetilde V_1| \le Ce^{-\eta L}
\end{equation}
and
\begin{equation}\label{wcon2}
|\widetilde W^{+}_1(L+\delta) -\widetilde W^{0+}_1(L+\delta)|\le Ce^{-\eta L}
\end{equation}
as $v_+\to 0$, uniformly in $\lambda$, where $C$, $\eta>0$ are
independent of $L$ and $\widetilde W_1^{0+}$ is the solution
of the limiting adjoint eigenvalue system
appearing in definition \eqref{duallimD} of $D^0$.
\end{corollary}

\begin{proof}[Proof of Lemma \ref{matching}]
First, make the independent coordinate change $x\to x-\delta$
normalizing the background wave to match the shock-wave case.
Making the dependent coordinate-change
\begin{equation}\label{dualL}
\tilde Z:= R^* \tilde W,
\end{equation}
$R$ as in \eqref{WZ}, reduces the adjoint equation
$\tilde W'=-A^*\tilde W$ to block lower-triangular form,
\begin{equation} \label{dualtri2}
\begin{aligned}
&\tilde Z'=-\tilde A^* \tilde Z=\\
&\,
\begin{pmatrix}
-\bar \lambda( J^T +
v_+\bb1 \theta_+ + \bb1 \tilde \Theta)^* &  0 \\
-\bar\lambda \bb1^T &
-f(\hat v)+\bar \lambda +\bar \lambda v_+ (\theta_+\bb1
+ \tilde \Theta \bb1)^* \\
\end{pmatrix}Z,
\end{aligned}
\end{equation}
with `` $\bar{ }$ '' denoting complex conjugate.

Denoting by $\tilde V^+_1$ a suitably normalized
element of the one-dimensional (slow) stable subspace
of $-\tilde A^*$, we find readily
(see \cite{HLZ} for further discussion)
that, without loss of generality,
\begin{equation}\label{dualVlim}
\tilde V^+_1 \to (0, 1, \bar \lambda (\gamma+\bar \lambda)^{-1})^T
\end{equation}
as $v_+\to 0$, while the associated eigenvalue $\tilde \mu_1^+\to 0,$
uniformly for $\lambda$ on an compact subset of $\R \lambda\ge 0$.
The dual mode $\tilde Z_1^+=R^* \tilde W_1^+$ is uniquely determined
by the property that it is asymptotic as $x\to +\infty$
to the corresponding constant-coefficient solution
$e^{\tilde \mu_1^+}\tilde V_1^+$
(the standard normalization of \cite{GZ,PZ,BHRZ}).

By lower block-triangular form \eqref{dualtri2}, the equations
for the slow variable $\tilde z^T:=(\tilde Z_1, \tilde Z_2)$ decouples
as a slow system
\begin{equation} \label{dualslowsys}
\tilde z'= -\Big(\lambda (J+v_+\bb1 \theta_+ + \bb1 \tilde \Theta) \Big)^*
\tilde z
\end{equation}
dual to \eqref{slowsys}, with solution operator
\begin{equation}\label{dualslowsoln}
P^*(x,\lambda)^{-1} e^{-\bar \lambda (J+v_+\bb1 \theta_+)^*)(x-y)}
P(y,\lambda)^{*}
\end{equation}
dual to \eqref{slowsoln}, i.e. (fixing $y=L$, say), solutions of general form
\begin{equation}\label{genform}
\tilde z(\lambda,x)=
P^*(x,\lambda)^{-1} e^{-\bar \lambda (J+v_+\bb1 \theta_+)^*)(x-y)}
\tilde v,
\end{equation}
$\tilde v \in \C^2$ arbitrary.

Denoting by
$$
\tilde Z_1^+(L):=R^*\tilde W_1^+(L),
$$
therefore, the unique (up to constant factor) decaying solution
at $+\infty$, and
$\tilde v_1^+:=((\tilde V_1^+)_1 , (\tilde V_1^+)_2)^T$,
we thus have evidently
$$
\tilde z_1^+(x,\lambda)=
P^*(x,\lambda)^{-1} e^{-\bar \lambda (J+v_+\bb1 \theta_+)^*)x}
\tilde v_1^+,
$$
which, as $v_+\to 0$, is uniformly bounded by
\begin{equation}\label{weakexp}
|\tilde z_1^+(x,\lambda)|\le Ce^{\epsilon x}
\end{equation}
for arbitrarily small $\epsilon>0$
and, by \eqref{dualVlim}, converges for $x$ less than or equal to
$X-\delta$ for any fixed $X$ simply to
\begin{equation}\label{simplelim}
\lim_{v_+\to 0}\tilde z_1^+(x,\lambda)=
P^*(x,\lambda)^{-1} (0,1)^T.
\end{equation}

Defining by $\tilde q:=(\tilde Z_1^+)_3$
the fast coordinate of $\tilde Z_1^+$, we have, by \eqref{dualtri2},
$$
\tilde q'
+\Big(f(\hat v)-\bar \lambda -(\lambda v_+ \theta_+\bb1 + \lambda
\tilde \Theta \bb1)^* \Big)
\tilde q=
\bar \lambda \bb1^T \tilde z_1^+,
$$
whence, by Duhamel's principle, any decaying solution is given by
$$
\tilde q(x,\lambda)=
\int_x^{+\infty} e^{\int_y^x a(z,\lambda, v_+)dz}\bar \lambda \bb1^T z_1^+(y) \, dy,
$$
where
$$
a(y,\lambda,v_+):=
-\Big(f(\hat v)-\bar \lambda -(\lambda v_+ \theta_+\bb1 + \lambda
\tilde \Theta \bb1)^* \Big).
$$
Recalling, for $\R \lambda \ge 0$, that $\R a \ge 1/2$, combining
\eqref{weakexp} and \eqref{simplelim},
and noting that $a$ converges uniformly on $y\le Y$ as $v_+\to 0$ for
any $Y>0$ to
$$
\begin{aligned}
a_0(y, \lambda)&:=
-f_0(\hat v)+\bar \lambda
+(\lambda\tilde \Theta_0 \bb1)^* \\
&=  (1+\bar \lambda)
+O(e^{-\eta y})
\end{aligned}
$$
we obtain by the Lebesgue
Dominated Convergence Theorem that
$$
\begin{aligned}
\tilde q(L,\lambda)&\to
\int_L^{+\infty} e^{\int_y^L a_0(z,\lambda)dz}\bar \lambda \bb1^T (0,1)^T \, dy\\
&=
\bar \lambda\int_L^{+\infty}
e^{(1+\bar \lambda)(L-y)+ \int_y^L O(e^{-\eta z})dz }
\, dy\\
&=
\bar \lambda
(1+\bar \lambda)^{-1}(1 +O(e^{-\eta L})).
\end{aligned}
$$
Recalling, finally, \eqref{simplelim}, and the fact that
$$
|P-Id|(L,\lambda), \,  |R-Id|(L,\lambda) \le Ce^{-\eta L}
$$
for $v_+$ sufficiently small, we obtain \eqref{wcon} as claimed.
\end{proof}

\begin{proof}[Proof of Corollary \ref{matching2}]
Again, make the coordinate change $x\to x-\delta$ normalizing
the background wave to match the shock-wave case.
Applying Proposition \ref{conjugation} to the limiting adjoint system
$$
\tilde W'=-(A^0)^* \tilde W=
\begin{pmatrix}0 & 0 & 0\\-\bar \lambda & 0 & 0\\
-1& -1 & 1+\bar \lambda \end{pmatrix}
\tilde W + O(e^{-\eta x})\tilde W,
$$
we find that, up to an $Id +O(e^{-\eta x})$ coordinate change,
$\tilde W_1^{0+}(x)$ is given by the exact solution
$\tilde W\equiv \tilde V_1$ of the limiting, constant-coefficient
system
$$
\tilde W'=-(A^0)^* \tilde W=
\begin{pmatrix}0 & 0 & 0\\-\bar \lambda & 0 & 0\\
-1& -1 & 1+\bar \lambda \end{pmatrix}
\tilde W.
$$
This yields immediately \eqref{limcon},
which, together with \eqref{wcon}, yields \eqref{wcon2}.
\end{proof}

\subsection{Convergence to $D^0$}\label{convergence}

The rest of our analysis is standard.

\begin{lemma}\label{regconj}
On $x\le L-\delta$ for any fixed $L>0$, there exists a coordinate-change
$W=TZ$ conjugating \eqref{evans_ode} to the limiting equations
\eqref{limevans_ode}, $T=T(x,\lambda, v_+)$, satisfying a uniform bound
\begin{equation}\label{Tbd}
|T-Id|\le C(L)v_+
\end{equation}
for all $v_+> 0$ sufficiently small.
\end{lemma}

\begin{proof}
Make the coordinate change $x\to x-\delta$ normalizing the background profile.
For $x\in (-\infty, 0]$, this is a consequence of the
{\it Convergence Lemma} of \cite{PZ}, a variation on
Proposition \ref{conjugation}, together with uniform
convergence of the profile and eigenvalue equations.
For $x\in [0,L]$, it is essentially continuous dependence;
more precisely, observing that
$ |A-A^0|\le C_1(L)v_+$ for $x\in [0,L]$,
setting $S:=T-Id$, and writing the
homological equation expressing conjugacy of \eqref{evans_ode}
and \eqref{limevans_ode}, we obtain
$$
S'- (AS-SA^0)= (A-A^0),
$$
which, considered as an inhomogeneous linear matrix-valued equation, yields
an exponential growth bound
\[
S(x)\le e^{Cx}(S(0)+ C^{-1}C_1(L)v_+)\]
for some $C>0$, giving the result.
\end{proof}

\begin{proof}[Proof of Theorem \ref{mainthm}: inflow case]
Make the coordinate change $x\to x-\delta$ normalizing the background profile.
Lemma \ref{regconj}, together with convergence as $v_+\to 0$
of the unstable subspace of $A_-$ to the unstable subspace
of $A^0_-$ at the same rate $O(v_+)$ (as follows by spectral separation
of the unstable eigenvalue of $A^0$ and standard matrix
perturbation theory) yields
\begin{equation}\label{Wbd}
|W_1^0(0,\lambda)- W_1^{00}(0,\lambda)|\le C(L)v_+.
\end{equation}

Likewise, Lemma \ref{regconj} gives
\begin{equation}\label{tildeWbd}
\begin{aligned}
|\tilde W_1^+(0,\lambda)- \tilde W_1^{0+}(0,\lambda)|&\le
C(L)v_+
|\tilde W_1^+(0,\lambda)|\\
&\quad +
|S_0^{L \to 0}|
|\tilde W_1^+(L,\lambda)- \tilde W_1^{0+}(L,\lambda)|,
\end{aligned}
\end{equation}
where $S_0^{y\to x}$ denotes the solution operator of
the limiting adjoint eigenvalue equation $\tilde W'=-(A^0)^*\tilde W$.
Applying Proposition \ref{conjugation} to the limiting system, we obtain
$$
|S_0^{L\to 0}|\le C_2e^{-A^0_+ L}\le C_2L|\lambda|
$$
by direct computation of $e^{-A^0_+ L}$, where $C_2$ is independent of $L>0$.
Together with \eqref{wcon2} and \eqref{tildeWbd}, this gives
$$
|\tilde W_1^+(0,\lambda)- \tilde W_1^{0+}(0,\lambda)|\le
C(L)v_+
|\tilde W_1^+(0,\lambda)| + L|\lambda|C_2Ce^{-\eta L},
$$
hence, for $|\lambda|$ bounded,
\begin{equation}\label{lastbd}
\begin{aligned}
|\tilde W_1^+(0,\lambda)- \tilde W_1^{0+}(0,\lambda)|&\le
C_3(L)v_+ |\tilde W_1^{0+}(0,\lambda)| + LC_4e^{-\eta L}\\
&\le
C_5(L)v_+  + LC_4e^{-\eta L}.\\
\end{aligned}
\end{equation}
Taking first $L\to \infty$ and then $v_+\to 0$, we obtain
therefore convergence of $W^+_1(0,\lambda)$ and $\tilde W_1^+(0,\lambda)$ to
$W^{0+}_1(0,\lambda)$ and $\tilde W_1^{0+}(0,\lambda)$, yielding the result
by definitions \eqref{adjevans} and \eqref{duallimD}.
\end{proof}

\begin{proof}[Proof of Theorem \ref{mainthm}: outflow case]
Straightforward, following the previous argument in the regular
region only.
\end{proof}

\subsection{Convergence to the shock case}\label{shocklim}

\begin{proof}[Proof of Theorem \ref{main2}: inflow case]
First make the coordinate change $x\to x-\delta$ normalizing
the background profile location to that of the shock wave case,
where $\delta \to +\infty$ as $v_0\to 1$.
By standard duality properties,
$$
D_{\rm in}=
W^0_1 \cdot \tilde W_1^+|_{x=x_0}
$$
is independent of $x_0$, so we may evaluate at $x=0$ as in the
shock case.
Denote by ${\mathcal{W}}^-_1$, $\tilde {\mathcal{W}}_1^+$
the corresponding modes in the shock case, and
$$
\mathcal{D}=
\mathcal{W}^-_1\cdot \tilde {\mathcal{W}}_1^+|_{x=0}
$$
the resulting Evans function.

Noting that $\tilde {\mathcal{W}}^1_+$
and $\tilde W^1_+$ are asymptotic to the unique stable mode
at $+\infty$ of the (same) adjoint eigenvalue equation,
but with translated decay rates, we see immediately that
$ \tilde {\mathcal{W}}^+_1=\tilde W^1_+ e^{-\delta \tilde \mu_1^+}.  $
$W^0_1 $ on the other hand is initialized at $x=-\delta$
(in the new coordinates $\tilde x=x-\delta$)
as
$$
W^0_1(-\delta)=(1,0,0)^T,
$$
whereas ${\mathcal{W}}^-_1$ is the unique unstable mode at $-\infty$
decaying as $e^{\mu_1^-x}V_1^-$, where $ V_1^- $ is the unstable
right eigenvector of
$$
A_-=
\begin{pmatrix}
0 & \lambda & \lambda\\
0 & 0 & \lambda\\
1& 1 & f(1)-\lambda\\
\end{pmatrix}.
$$

Denote by $\tilde V^-_1$ the associated dual unstable left eigenvector and 
$$
\Pi^-_1:=V^-_1(\tilde V^-_1)^T
$$
the eigenprojection onto the stable vector $V^-_1$.
By direct computation,
$$
\tilde V^-_1=c(\lambda)(1, 1+\lambda/\mu^-_1,\mu^-_1)^T, 
\quad c(\lambda)\ne 0,
$$
yielding
\begin{equation}\label{goodproj}
\Pi^-_1 W^0_1=:\beta(\lambda)=c(\lambda)\ne 0
\end{equation}
for $\Re \lambda\ge 0$, on which $\Re \mu^-_1>0$.

Once we know \eqref{goodproj}, we may finish by a standard argument,
concluding by exponential attraction
in the positive $x$-direction of the unstable mode that other modes
decay exponentially as $x\to 0$, leaving the
contribution from $\beta(\lambda)V_1^-$ plus a negligible $O(e^{-\eta \delta})$
error, $\eta>0$, from which we may conclude that
${\mathcal{W}}^-_1|_{x=0}\sim \beta^{-1}e^{-\delta \mu_1^-} W^0_1|_{x=0}.$
Collecting information, we find that
$$
\mathcal{D}(\lambda)=
\beta(\lambda)^{-1}
e^{-\delta
(\bar \mu_1^-+ \tilde \mu_1^+)
(\lambda)} D_{\rm in}(\lambda) + O(e^{-\eta \delta}),
$$
$\eta>0$, yielding the claimed convergence as $v_0\to 1$, $\delta\to +\infty$.
\end{proof}

\begin{proof}[Proof of Theorem \ref{main2}: outflow case]
For $\lambda$ uniformly bounded from zero,
$\tilde W^0_1=(0,-1, -\bar \lambda /(\bar \lambda -\hat v'(0)))^T$
converges uniformly as $v_0\to 0$ to
$$
(0,-1, -1)^T,
$$
whereas the shock Evans function $\mathcal{D}$ is initiated by
$\tilde{\mathcal{ W}}^+_1$ proportional to
$$
\tilde{\mathcal{V}}^+_1= (0,-1, -1-\lambda)^T
$$
agreeing in the first two coordinates with $\tilde W^0_1$.
By the boundary-layer analysis of Section \ref{estW1}, the
backward (i.e., decreasing $x$) evolution of the adjoint
eigenvalue ODE reduces in the asymptotic
limit $v_+\to 0$ (forced by $v_0\to 0$) to a decoupled slow flow
$$
\tilde w'=\begin{pmatrix}0 & \bar \lambda\\ 0 & 0\end{pmatrix}w,
\qquad w\in \C^2
$$
in the first two coordinates, driving an exponentially slaved
fast flow in the third coordinate.
From this, we may conclude that solutions agreeing in the first
two coordinates converge exponentially as $x$ decreases.
Performing an appropriate normalization, as in the inflow case
just treated, we thus obtain the result.
We omit the details, which follow what has already been done
in previous cases.
\end{proof}

\subsection{The stability index}\label{stabsection}
Following \cite{SZ,GMWZ.5}, we note that $D_{\rm in}(\lambda)$ is
real for real $\lambda$, and nonvanishing for real $\lambda$
sufficiently large, hence $\sgn D_{\rm in}(+\infty)$ is well-defined
and constant on the entire (connected) parameter range. The number
of roots of $D_{\rm in}$ on $\Re \lambda \ge 0$ is therefore even or
odd depending on the {\it stability index}
$$
\sgn [D_{\rm in}(0)D_{\rm in}(+\infty)].
$$
Similarly, recalling that $D_{\rm out}(0)\equiv 0$, we find that the
number of roots of $D_{\rm out}$ on $\Re \lambda \ge 0$ is even or
odd depending on
$$
\sgn [D_{\rm out}'(0)D_{\rm out}(+\infty)].
$$

\begin{proof}[Proof of Lemma \ref{index}: inflow case]
Examining the adjoint equation at $\lambda=0$,
$$
\tilde W'=-A^*\tilde W,
\qquad
-A^*(x,0)=
\begin{pmatrix}
0 & 0 & -\hat v\\
0 & 0 & -\hat v\\
0 & 0 & -f(\hat v)\\
\end{pmatrix},
$$
$-f(v_+)>0$, we find by explicit computation that
the only solutions that are bounded as $x\to +\infty$ are
the {\it constant solutions} $\tilde W\equiv (a,b,0)^T$.
Taking the limit $V^+_1(0)$ as $\lambda \to 0^+$ along the real axis
of the unique stable eigenvector of $-A^*_+(\lambda)$,
we find (see, e.g., \cite{Z.3}) that it lies in the direction
$(1, 2+a_j^+,0)^T$, where $a_j^+>0$ is the positive characteristic
speed of the hyperbolic convection matrix
$\begin{pmatrix}
1 & - 1\\
-h(v_+)/v_+^{\gamma+1} & 1\\
\end{pmatrix}$, i.e.,
$V_1^-= c(v_0,v_+)(1, 2+a_j^+,0)^T$, $c(v_0,v_+)\ne 0$. Thus,
$D_{\rm in}(0)= V_1^-\cdot (1,0,0)^T=\bar c(v_0,v_+) \ne 0$ as
claimed. On the other hand, the same computation carried out for
$D^0_{\rm in}(0)$ yields $D^0_{\rm in}(0)\equiv 0$. (Note: $a_j\sim
v_+^{-1/2}\to +\infty$ as $v_+\to 0$.)
Similarly, as $v_0\to 0$,
$$
D^0_{\rm in} (\lambda)\to (1,0,0)^T\cdot
(0,1,*)^T\equiv 0.
$$ Finally, note $D_{\rm in}(0)\ne 0$ implies that the stability
index, since continuously varying so long as it doesn't vanish and
taking discrete values $\pm 1$, must be constant on the connected
set of parameter values.  Since inflow boundary layers are known to
be stable on some part of the parameter regime by energy estimates
(Theorem \ref{main2}), we may conclude that the stability index is
identically one and therefore there are an even number of unstable
roots for all $1>v_0\ge v_+>0$.

To establish that $(D^0_{\rm in})'(0)\ne 0$, we compute
$$
D^0_{\rm in}\; '(0) = (\partial_\lambda W_1^{00})\cdot
\widetilde{W}_1^{0+} + W_1^{00}\cdot (\partial_\lambda
\widetilde{W}_1^{0+}).
$$ Since $W_1^{00} \equiv (1,0,0)$ is independent of $\lambda$, we
need only show that the first component of $\partial_\lambda
\widetilde{W}_1^{0+}$ is nonzero. Note that $\partial_\lambda
W_1^{0+}$ solves the limiting adjoint variational equations

$$
(\partial_\lambda \widetilde{W}_1^{0+})'(0) + (A^0)^*(x,0)
\partial_\lambda \widetilde{W}_1^{0+} =  b(x)
$$ with
$$
(A^0)^*(x,0)\tilde =
\begin{pmatrix}
0 & 0 & \hat v^0\\
0 & 0 & \hat v^0\\
0 & 0 & f^0(\hat v^0)\\
\end{pmatrix}, \qquad b(x) = \begin{pmatrix}
0 \\
\hat v^0 + \hat u^0 - \frac{\hat v^0\, '}{\hat v^0} - 1 \\
3\hat v^0 - \frac{\hat v^0\, '}{\hat v^0} - 1 \end{pmatrix}.
$$ By \eqref{tildeV}, and the fact that 
$\partial_\lambda \tilde \mu_1^{0+}\equiv 0$,
 $\partial_\lambda \widetilde{W}_1^{0+}(x)$ is chosen so
that asymptotically at $x=+\infty$ it lies in the direction of
$\partial_\lambda \tilde V_1 = (0,0,-1)$. Set $\partial_\lambda
\widetilde{W}_1^{0+} = (\partial_\lambda \widetilde{W}_{1,\;1}^{0+},
\partial_\lambda \widetilde{W}_{1,\;2}^{0+}, \partial_\lambda
\widetilde{W}_{1,\;3}^{0+})^T$.  Then the third component solves
$$
(\partial_\lambda \widetilde{W}_{1,\;3}^{0+})' + \hat v ^0
\partial_\lambda \widetilde{W}_{1,\; 3}^{0+} =b_3 := 3\hat v^0 -
\frac{\hat v^0\, '}{\hat v^0} - 1
$$ which has solution
$$
\partial_\lambda \widetilde{W}_{1,\;3}^{0+}(x) = \partial_\lambda \widetilde{W}_{1,\;3}^{0+}(+\infty)\varphi(x) -
\varphi(x) \int_x^\infty \varphi^{-1}(y) b_3(y)dy
$$ where
$$
\varphi(x) = e^{\int_x^\infty \hat v^0(y) dy}.
$$  Integrating the equation for the first component of $\partial_\lambda \widetilde{W}_1^{0+}$ yields
\begin{align*}
\begin{split}
\partial_\lambda \widetilde{W}_{1,\;1}^{0+}(x) & = \partial_\lambda
\widetilde{W}_{1,\;1}^{0+}(+\infty) + \int_x^\infty \partial_\lambda
\widetilde{W}_{1,\;3}^{0+}(y) dy \\
&= \partial_\lambda \widetilde{W}_{1,\;1}^{0+}(+\infty) +
\partial_\lambda \widetilde{W}_{1,\;3}^{0+}(+\infty) \int_x^\infty
\hat v^0(y) \varphi(y) dy \\
& - \int_x^\infty \left( \varphi(y) \int_y^\infty \varphi^{-1}(z)
b_3(z) dz \right) dy.
\end{split}
\end{align*} Using the condition $\partial_\lambda \widetilde{W}_{1}^{0+}(+\infty) =
(0,0,-1)^T$ we have $\partial_\lambda
\widetilde{W}_{1,\;1}^{0+}(+\infty) = 0, \partial_\lambda
\widetilde{W}_{1,\;3}^{0+}(+\infty) = -1$ so that

$$
\partial_\lambda \widetilde{W}_{1,\;1}^{0+}|_{x=0} = - \int_0^\infty \hat v^0(y) \varphi(y) dy - \int_x^\infty \left(
\varphi(y) \int_y^\infty \varphi^{-1}(z) b_3 dz \right) dy.
$$  Finally, note that by using \eqref{v^0} we have $b_3=1-\tanh(\frac{x-\delta}{2})$ so that for all $x\geq 0$, $\varphi(x), b_3(x) \geq
0$ which implies
$$
D^0_{\rm in}\, '(0) = \partial_\lambda
\widetilde{W}_{1,\;1}^{0+}|_{x=0} \neq 0.
$$
\end{proof}

\begin{remark}\label{apparent}
The result $D_{\rm in}(0)\ne 0$ at first sight appears to contradict that
of Theorem \ref{main3}, since $\mathcal{D}(0)=0$ for
the shock wave case.  This apparent contradiction is explained
by the fact that the normalizing factor
$e^{-\delta (\bar \mu_1^- + \tilde \mu_1^+) }$
is exponentially decaying in $\delta$ for $\lambda=0$, since
$\tilde \mu_1^+(0)=0$, while $\Re\mu^-_1>0$.
Recalling that $\delta\to +\infty$ as $v_0\to 1$, we recover the result
of Theorem \ref{main3}.
\end{remark}

\begin{proof}[Proof of Lemma \ref{index}: outflow case]
Similarly, we compute
$$
D'_{\rm out}(0)=
\partial \lambda W^-_1 \cdot \tilde W^0_1,
$$
where $ \partial \lambda W^-_1|_{\lambda=0}$ satisfies the variational
equation $L\partial_\lambda U^-_1(0)=U^-_1=\hat U'$, or, written
as a first-order system,
$$
(\partial \lambda W^-_1)'- A(x,0) \partial \lambda W^-_1=
\begin{pmatrix}\hat u_x\\ \hat v_x\\ -\hat v_x \end{pmatrix},
\qquad
A(x,0)=
\begin{pmatrix} 0 & 0 & 0\\
0 & 0 & 0\\
\hat v & \hat v & f(\hat v)\\
\end{pmatrix},
$$
which may be solved exactly for the unique solution
decaying at $-\infty$ of
$$
W^-_1(0)=
\begin{pmatrix} 0 \\ 0\\ \hat v' \end{pmatrix},
\qquad
(\partial \lambda W^-_1)(0)=
\begin{pmatrix}\hat u- u_-\\ \hat v- v_-\\ * \end{pmatrix}.
$$
Recalling from \eqref{tildew1} that
$\widetilde{W}^0_1(\lambda)= (0, -1, -\lambda/(\lambda -\hat v'(0)))^T$,
hence
$$
\widetilde{W}^0_1(0)= (0, -1, 0)^T,
\qquad
\partial_\lambda \widetilde{W}^0_1(0)= (0, 0, 1/\hat v'(0))^T,
$$
we thus find that
$$
\begin{aligned}
D_{\rm out}'(0)&=\partial_\lambda W^-_1(0)\cdot \widetilde{W}^0_1(0)
+W^-_1(0)\cdot \partial_\lambda\widetilde{W}^0_1(0)
\\
&= -(\hat v(0)-1)+1 =2-v_0\ne 0
\end{aligned}
$$
as claimed.  The proof that $(D^0_{\rm out})'(0)\ne 0$ goes
similarly.

Finally, as in the proof of the inflow case, we note that nonvanishing
implies that the stability index is constant across the entire
(connected) parameter range, hence we may conclude that it is identically
one by existence of a stable case (Corollary \ref{v01}), and
therefore that the number of nonzero unstable roots is even, as claimed.
\end{proof}

\subsection{Stability in the shock limit}\label{smallv0}

\begin{proof}[Proof of Corollary \ref{v01}: inflow case]
By Proposition \ref{redenergy}
we find that $D_{\rm in}$
has at most a single zero in $\Re \lambda \ge 0$.  However,
by our stability index results, Theorem \ref{index}, the
number of eigenvalues in $\Re \lambda \ge 0$ is even.
Thus, it must be zero, giving the result.
\end{proof}

\begin{proof}[Proof of Corollary \ref{v01}: outflow case]
By Theorem \ref{main3}, $D_{\rm out}$, suitably renormalized,
converges as $v_0\to 0$ to the Evans function for the (unintegrated)
shock wave case.
But, the shock Evans function by the results of \cite{BHRZ,HLZ}
has just a single zero at $\lambda=0$ on $\Re \lambda\ge 0$,
already accounted for in $D_{\rm out}$
by the spurious root at $\lambda=0$ introduced by recoordinatization
to ``good unknown''.
\end{proof}

\subsection{Stability for small $v_0$}\label{corner}

Finally, we treat the remaining,
``corner case'' as $v_+$, $v_0$ simultaneously approach zero.  
The fact (Lemma \ref{index}) that
$$
\lim_{v_0\to 0}\lim_{v_+\to 0}D_{\rm in}(\lambda)\equiv 0
$$
shows that this limit is quite delicate; indeed, this
is the most delicate part of our analysis.

\begin{proof}[Proof of Theorem \ref{main2}: inflow case]
Consider again the adjoint system 
$$
\tilde W'=-A^*(x,\lambda)\tilde W,
\qquad
A^*(x,\lambda)\tilde =
\begin{pmatrix}
0 & 0 & \hat v\\
\bar \lambda & 0 & \hat v\\
\bar \lambda & \bar \lambda & f(\hat v) - \bar \lambda\\
\end{pmatrix}.
$$
By the boundary analysis of Section \ref{estW1},
$$
\tilde W=
\Big(\alpha, 1, \frac{\alpha \tilde \mu-\bar \lambda(\alpha +1)}
{-f(\hat v)+\bar \lambda}\Big)^T
+ O(e^{-\eta |x-\delta|}),
$$
where $\alpha:=\frac{\tilde \mu_+}{\tilde \mu_++ \bar \lambda}$,
and $\tilde \mu$ is the unique stable eigenvalue of $A^*_+$,
satisfying (by matrix perturbation calculation)
$$
\tilde \mu= \bar \lambda(v_+^{1/2} + O(v_+))
$$
and thus $\alpha =v_+^{1/2}+O(v_+)$
as $v_0\to 0$ (hence $v_+\to 0$) on bounded subsets of $\Re \lambda \ge 0$.
Combining these expansions, we have
$$
\tilde W_1(+\infty)=v_+^{1/2}(1+o(1)),
\qquad
\tilde W_3=
\frac{-\bar \lambda}
{-f(\hat v)+\bar \lambda}
(1 + o(1))
$$
for $v_0$ sufficiently small.

From the $\tilde W_1$ equation $\tilde W'= \hat v \tilde W_3$,
we thus obtain
$$
\begin{aligned}
\tilde W_1(0)&= 
\tilde W_1(+\infty)- \int_0^{+\infty}\hat v \tilde W_3(y)\, dy\\
&= (1+o(1))\times
\Big(
v_+^{1/2} + \int_0^{+\infty}
\frac{\bar \lambda \hat v}{-f(\hat v)+\bar \lambda}(y)\, dy \Big).
\end{aligned}
$$
Observing, finally, that, for $\Re \lambda \ge 0$, the
ratio of real to imaginary parts of 
$\frac{\bar \lambda \hat v}{-f(\hat v)+\bar \lambda}(y)$ is
uniformly positive, we find that $\Re \tilde W_1(0)\ne 0$ 
for $v_0$ sufficiently small, which yields nonvanishing of 
$D_{\rm in}(\lambda)$ on $\Re \lambda \ge 0$ as claimed.
\end{proof}

\section{Numerical computations}\label{computations}

In this section, we show, through a systematic numerical Evans function
study, that there are no unstable eigenvalues for
\[
(\gamma,v_+) \in[1,3]\times(0,1],
\]
in either inflow or outflow cases.  As defined in Section \ref{evanssec},
the Evans function is analytic in the right-half plane and reports a value
of zero precisely at the eigenvalues of the linearized operator
\eqref{linearized}.  Hence we can use the argument principle to determine
if there are any unstable eigenvalues for this system.  Our approach
closely follows that of \cite{BHRZ,HLZ} for the shock case with only two
major differences.  First, our shooting algorithm is only one sided as we
have the boundary conditions \eqref{w01} and \eqref{tildew1} for the
inflow and outflow cases, respectfully.  Second, we ``correct'' for the
displacement in the boundary layer when $v_0\approx 1$ in the inflow case
and $v_0\approx 0$ in the outflow case so that the Evans function
converges to the shock case as studied in \cite{BHRZ,HLZ} (see discussion
in Section \ref{sectionshocklimit}).

\begin{figure}[t]
\begin{center}
$\begin{array}{cc}
\includegraphics[width=6.4cm,height=4.5cm]{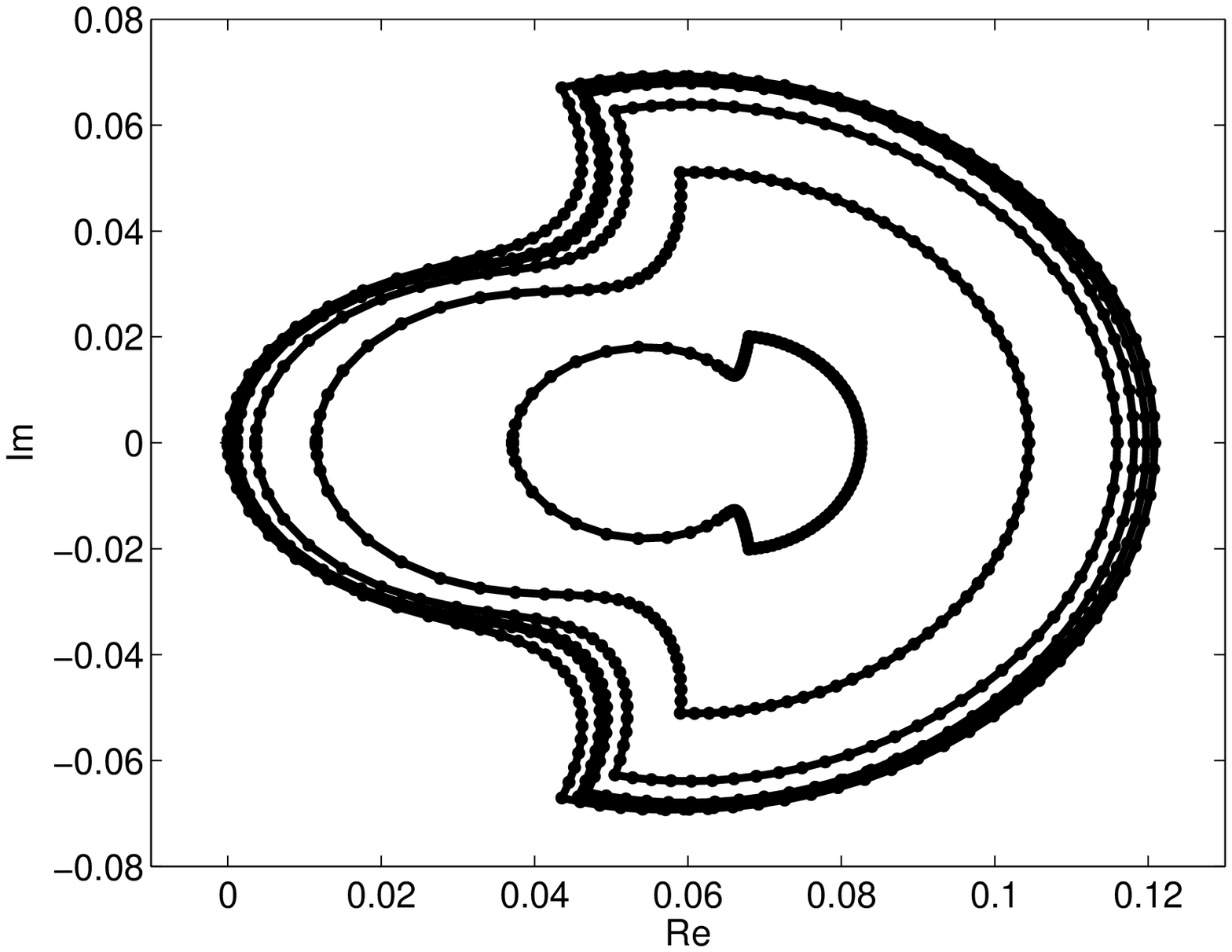} &
\includegraphics[width=6.4cm,height=4.5cm]{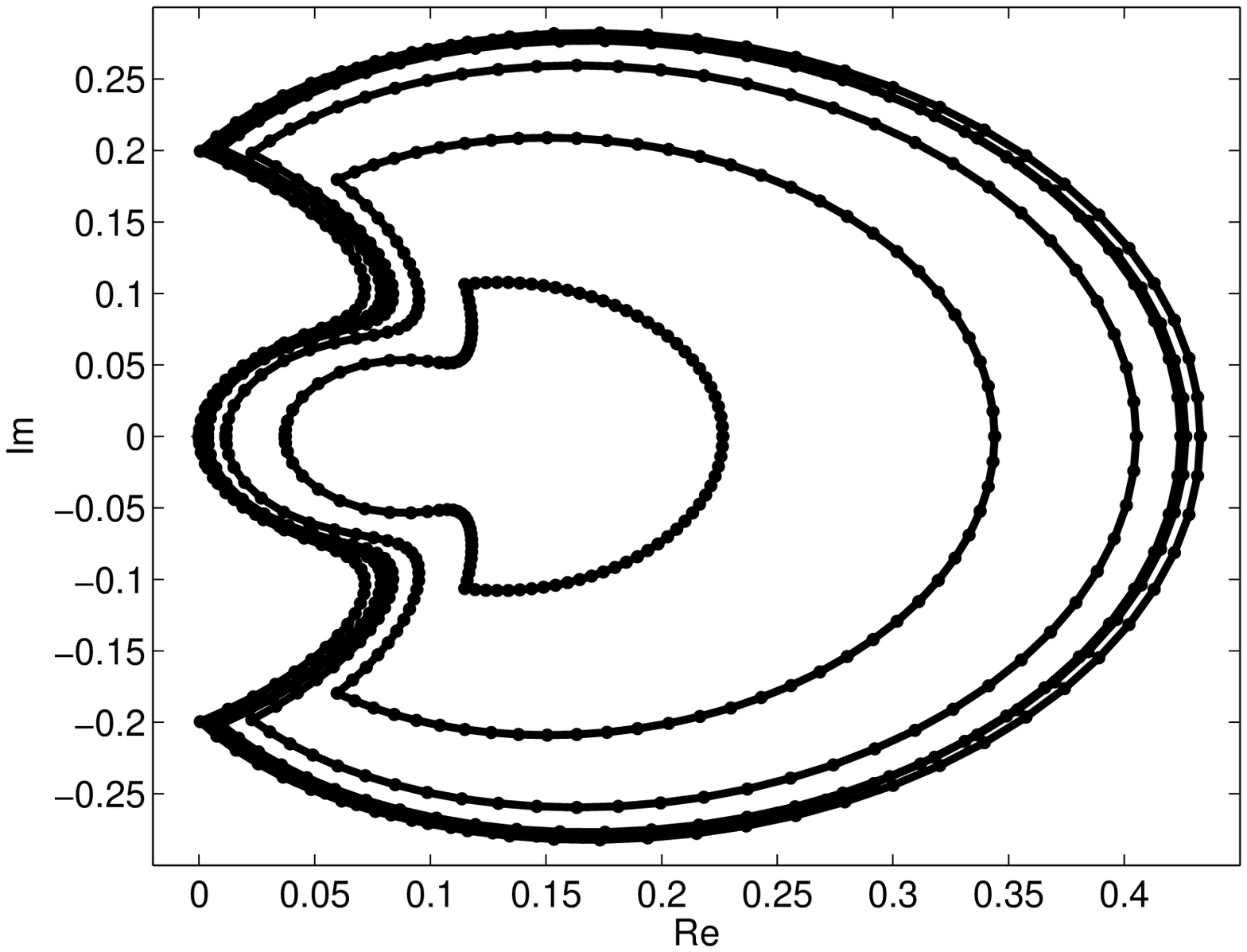} \\
\mbox{\bf (a)} & \mbox{\bf (b)}\\
\includegraphics[width=5.75cm]{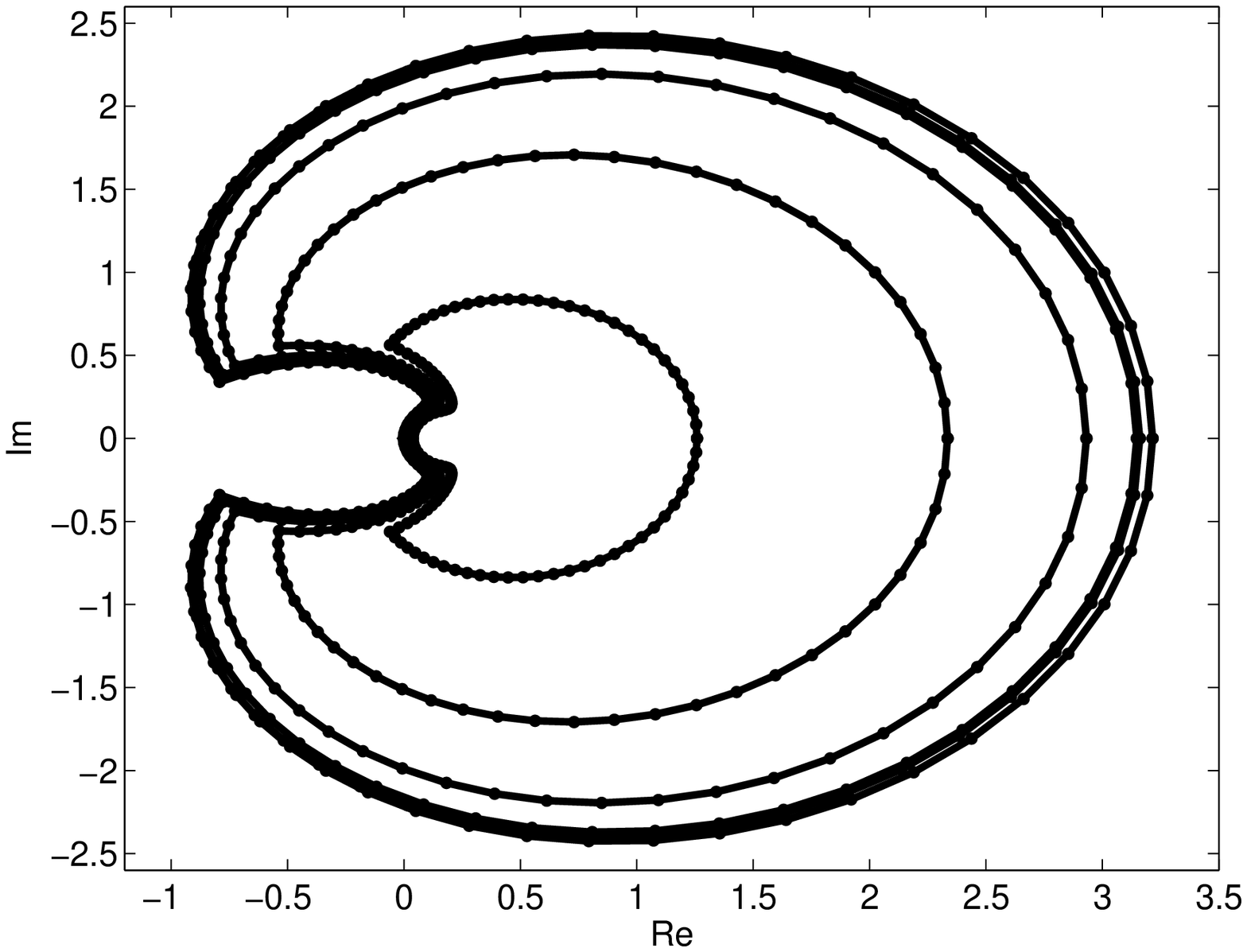} &
\includegraphics[width=5.6cm]{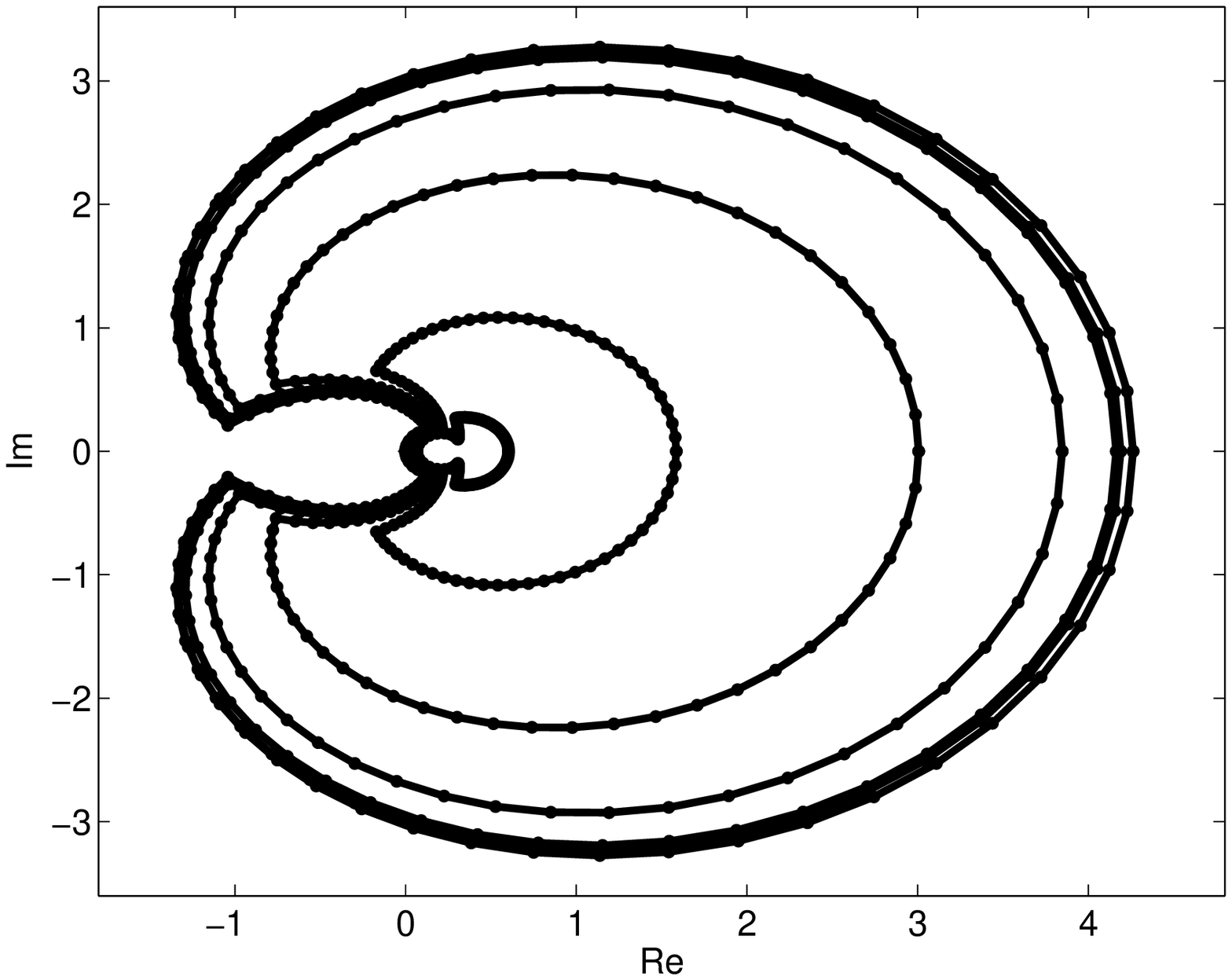} \\
\mbox{\bf (c)} & \mbox{\bf (d)}
\end{array}$
\end{center}
\caption{Typical examples of the inflow case, showing convergence to the
limiting Evans function as $v_+\to 0$ for a monatomic gas, $\gamma=5/3$,
with $(a)$ $v_0=0.1$, $(b)$ $v_0=0.2$, $(c)$ $v_0=0.4$, and $(d)$
$v_0=0.7$.  The contours depicted, going from inner to outer, are images
of the semicircle $\phi$ under $D$ for
$v_+=1e\!-\!2,1e\!-\!3,1e\!-\!4,1e\!-\!5,1e\!-\!6$, with the outer-most
contour given by the image of $\phi$ under $D^0$, that is, when $v_+= 0$. 
Each contour consists of $60$ points in $\lambda$.}
\label{first}
\end{figure}

The profiles were generated using Matlab's {\tt bvp4c} routine, which is
an adaptive Lobatto quadrature scheme.  The shooting portion of the Evans
function computation was performed using Matlab's {\tt ode45} package,
which is the standard 4th order adaptive Runge-Kutta-Fehlberg method
(RKF45).  The error tolerances for both the profiles and the shooting were
set to {\tt AbsTol=1e-6} and {\tt RelTol=1e-8}.  We remark that Kato's ODE
(see Section \ref{evanssec} and \cite{Kato,HSZ} for details) is used to
analytically choose the initial eigenbasis for the stable/unstable
manifolds at the numerical values of infinity at $L=\pm 18$.  Finally in Section 
\ref{errorstudy}, we carry out a numerical convergence study similar to that
in \cite{BHRZ}.

\begin{figure}[t]
\begin{center}
$\begin{array}{cc}
\includegraphics[width=5.75cm]{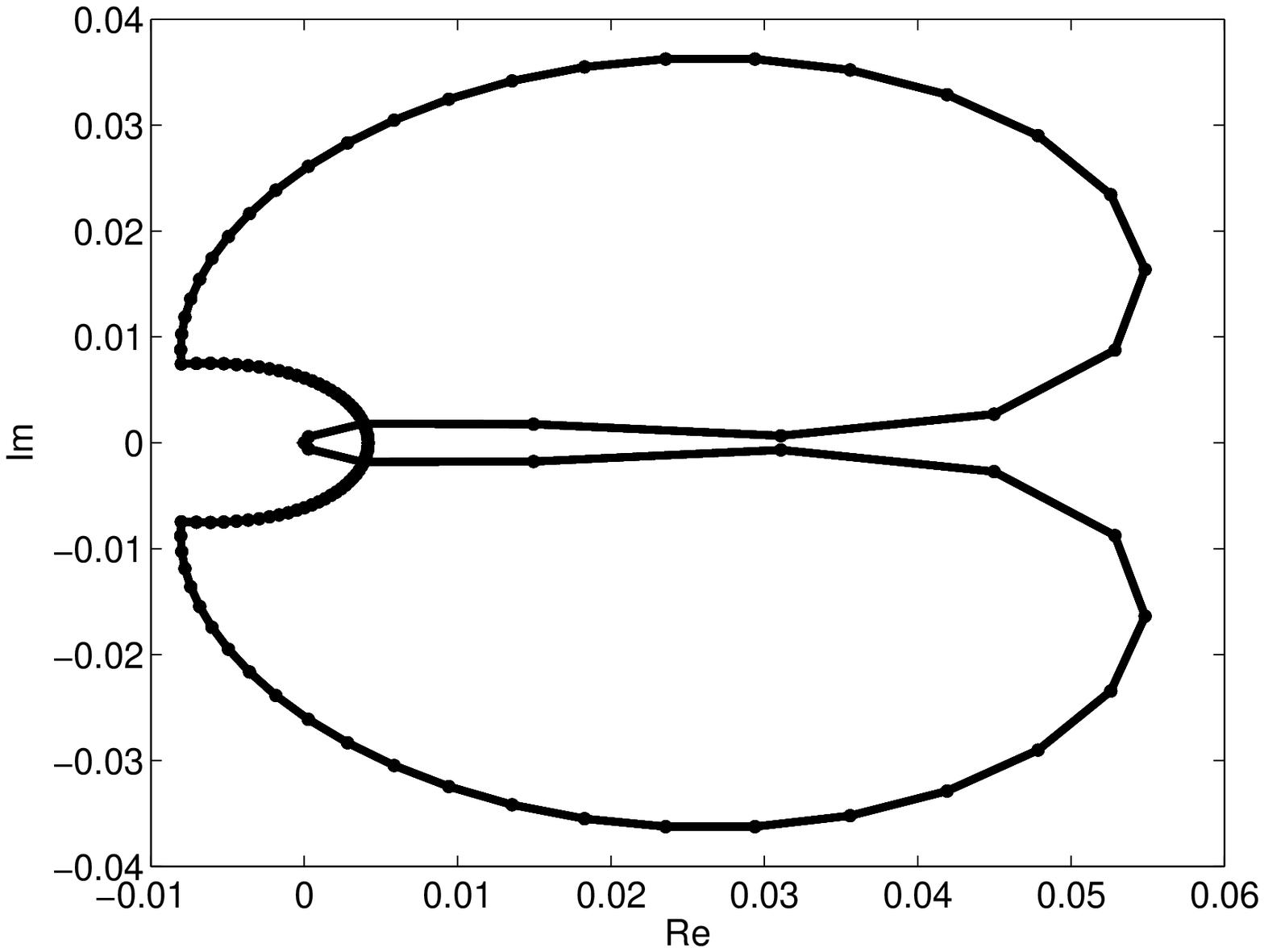} &
\includegraphics[width=5.75cm]{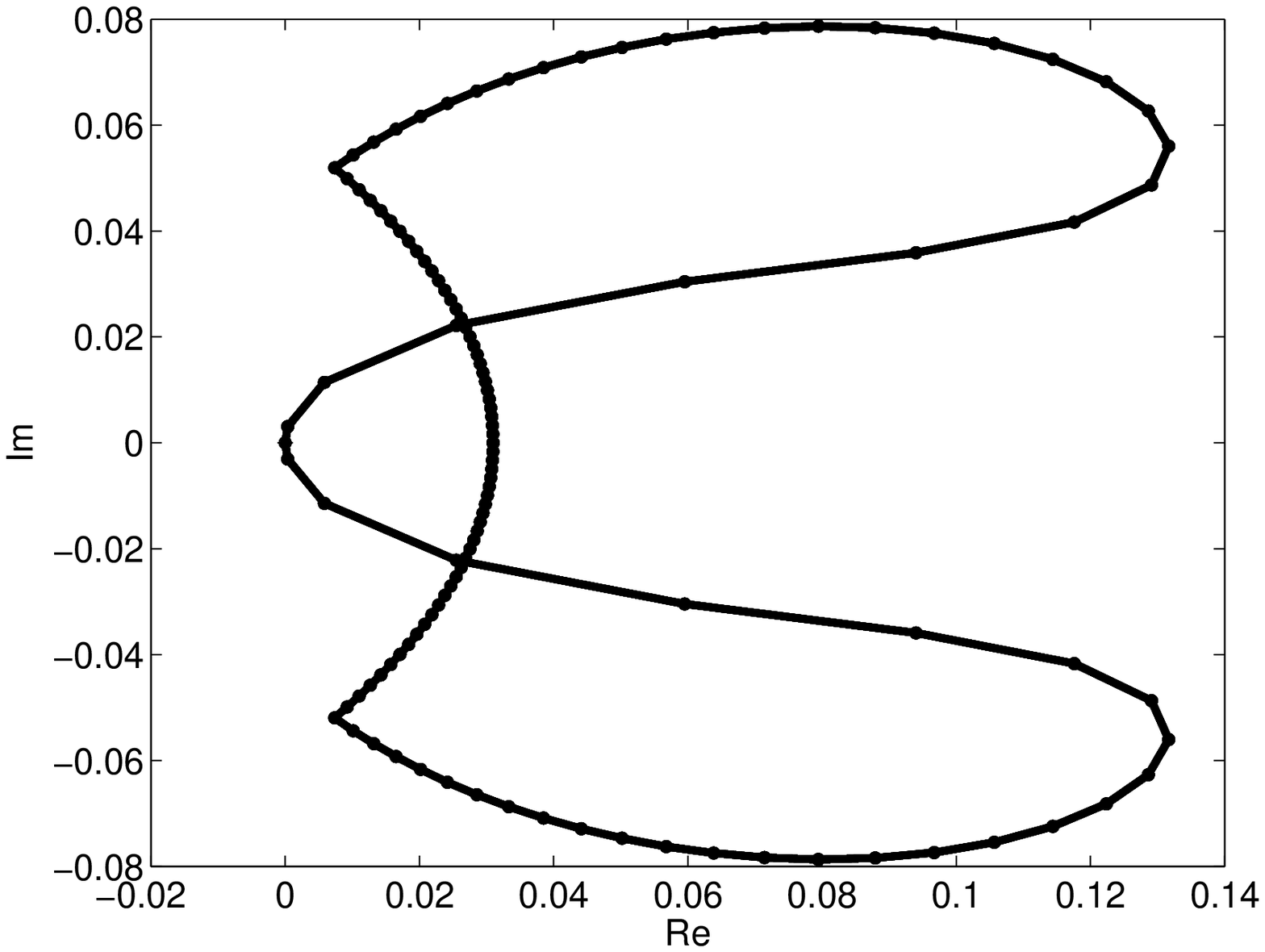} \\
\mbox{\bf (a)} & \mbox{\bf (b)}\\
\includegraphics[width=5.75cm]{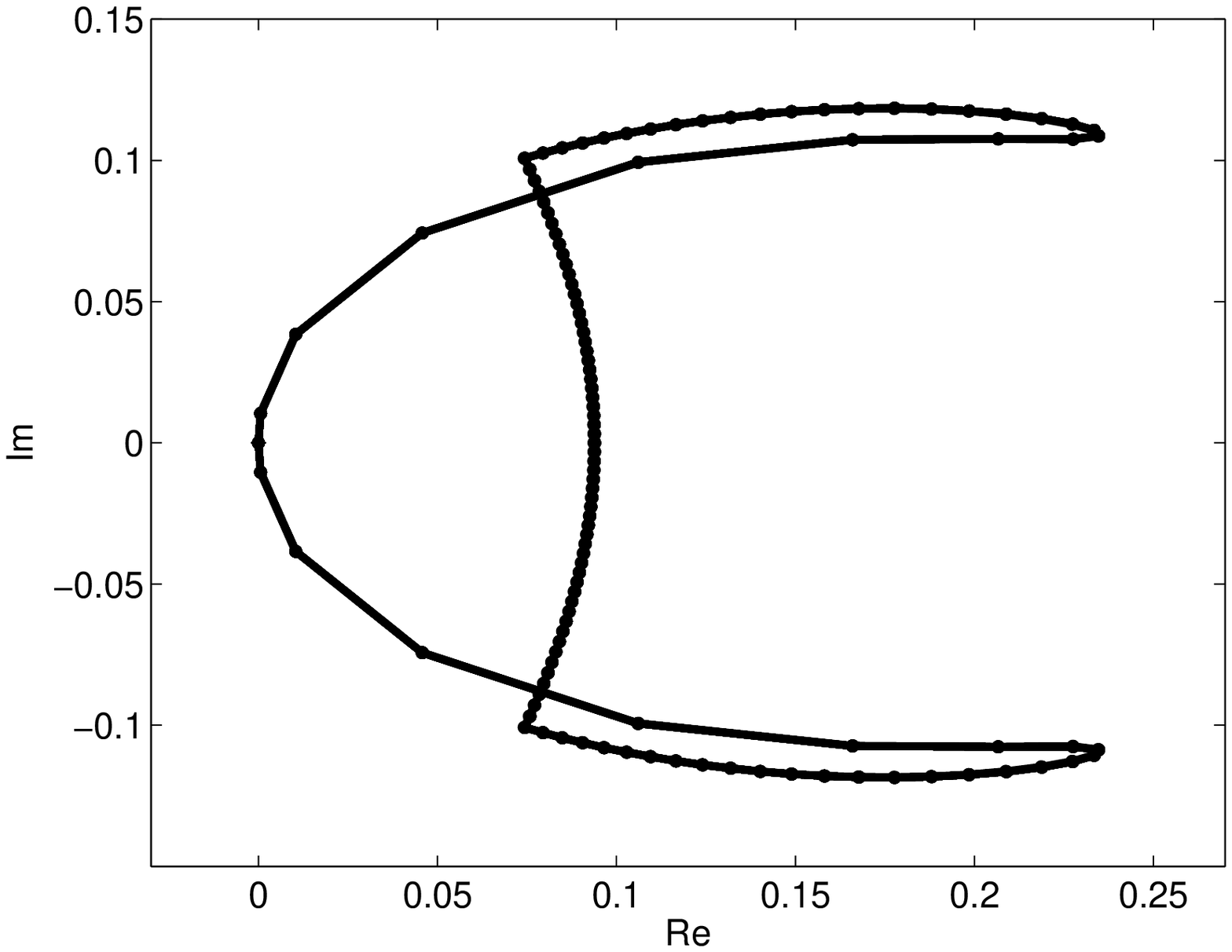} &
\includegraphics[width=5.3cm]{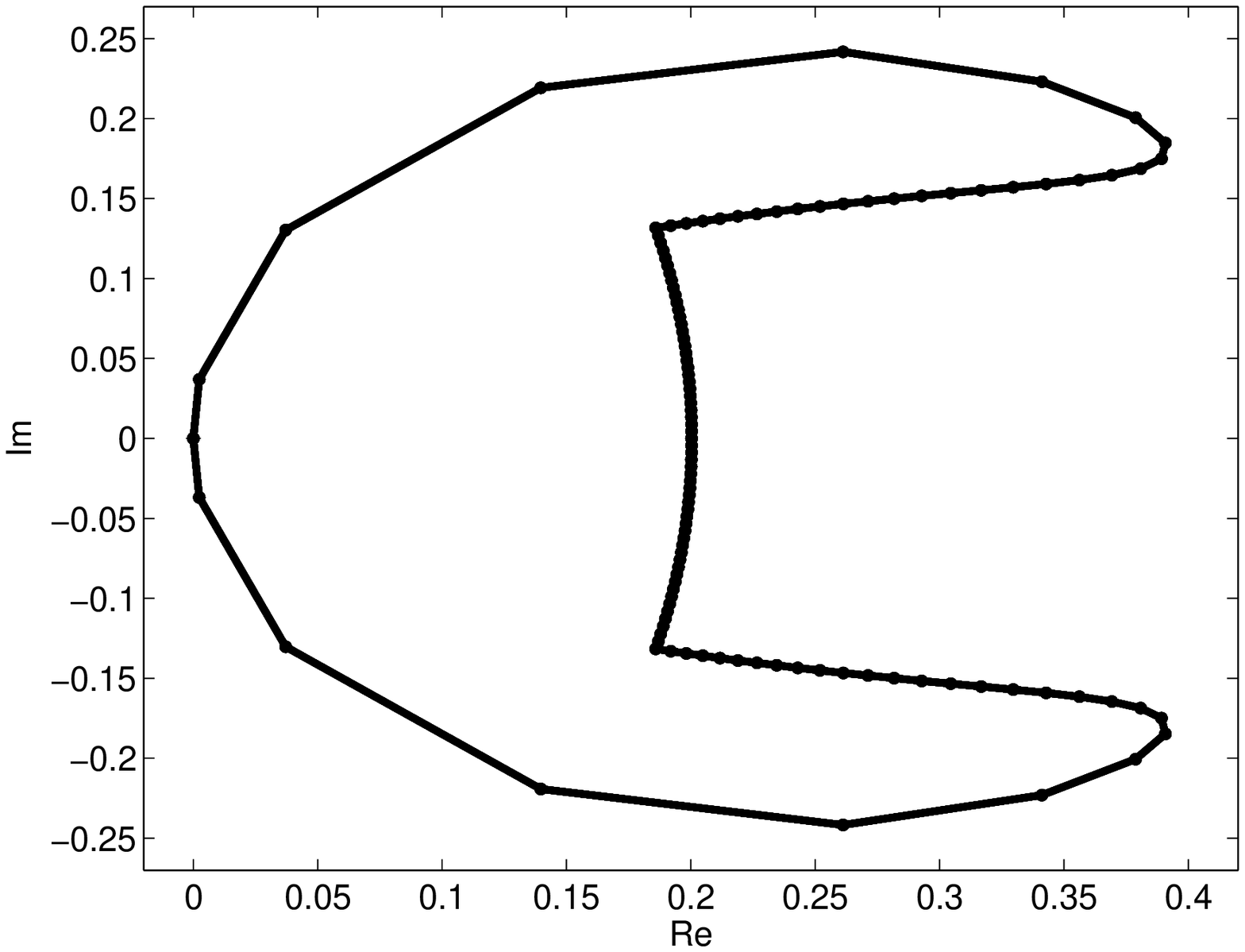} \\
\mbox{\bf (c)} & \mbox{\bf (d)}
\end{array}$
\end{center}
\caption{Typical examples of the outflow case, showing convergence to the
limiting Evans function as $v_+\to 0$ for a monatomic gas, $\gamma=5/3$,
with $(a)$ $v_0=0.2$, $(b)$ $v_0=0.4$, $(c)$ $v_0=0.6$, and $(d)$
$v_0=0.8$.  The contours depicted are images of the semicircle $\phi$
under $D$ for $v_+=1e\!-\!2,1e\!-\!3,1e\!-\!4,1e\!-\!5,1e\!-\!6$, and the
limiting case $v_+= 0$.  Interestingly the contours are essentially
(visually) indistinguishable in this parameter range. Each contour
consists of $60$ points in $\lambda$}
\label{second1}
\end{figure}

\subsection{Winding number computations}

The high-frequency estimates in Proposition \ref{hf} restrict the set of
admissible unstable eigenvalues to a fixed compact triangle $\Lambda$ in
the right-half plane (see \eqref{hfbounds1} and \eqref{hfbounds2} for the
inflow and outflow cases, respectively).  We reiterate the remarkable
property that $\Lambda$ does not depend on the choice of $v_+$ or $v_0$. 
Hence, to demonstrate stability for a given $\gamma$, $v_+$ and $v_0$, it
suffices to show that the winding number of the Evans function along a
contour containing $\Lambda$ is zero.  Note that in our region of
interest, $\gamma\in[1,3]$, the semi-circular contour given by
\[
\phi:=\partial(\{\lambda\mid \R\geq 0\}\cup\{\lambda\mid |\lambda|\leq 10\}),
\]
contains $\Lambda$ in both the inflow and outflow cases.  Hence, for
consistency we use this same semicircle for all of our winding number
computations.

A remarkable feature of the Evans function for this system, and one that
is shared with the shock case in \cite{BHRZ,HLZ}, is that the Evans function 
has limiting behavior as the amplitude increases, Section \ref{analytical}.  For
the inflow case, we see in Figure \ref{first}, the mapping of the contour $\phi$ 
for the monatomic case ($\gamma=5/3$), for several different choices of $v_0$, 
as $v_+\rightarrow 0$.  We remark that the winding numbers for $0\leq v_+\leq 1$ 
are all zero, and the limiting contour touches zero due to the emergence of a 
zero root in the limit.  Note that the limiting case contains the contours of all other 
amplitudes.  Hence, we have spectral stability for all amplitudes.

The outflow case likewise has a limiting behavior, however, all contours
cross through zero due to the eigenvalue at the origin.  Nonetheless,
since the contours only wind around once, we can likewise conclude that
these profiles are spectrally stable.  We remark that the outflow case
converges to the limiting case faster than the inflow case as is clear
from Figure \ref{second1}.
Indeed, $v_+=1e\!-\!2$ and the limiting case $v_+=0$, as well as all of
the values of $v_+$ in between, are virtually indistinguishable.

\begin{figure}[t]
\begin{center}
$\begin{array}{cc}
\includegraphics[width=5.75cm]{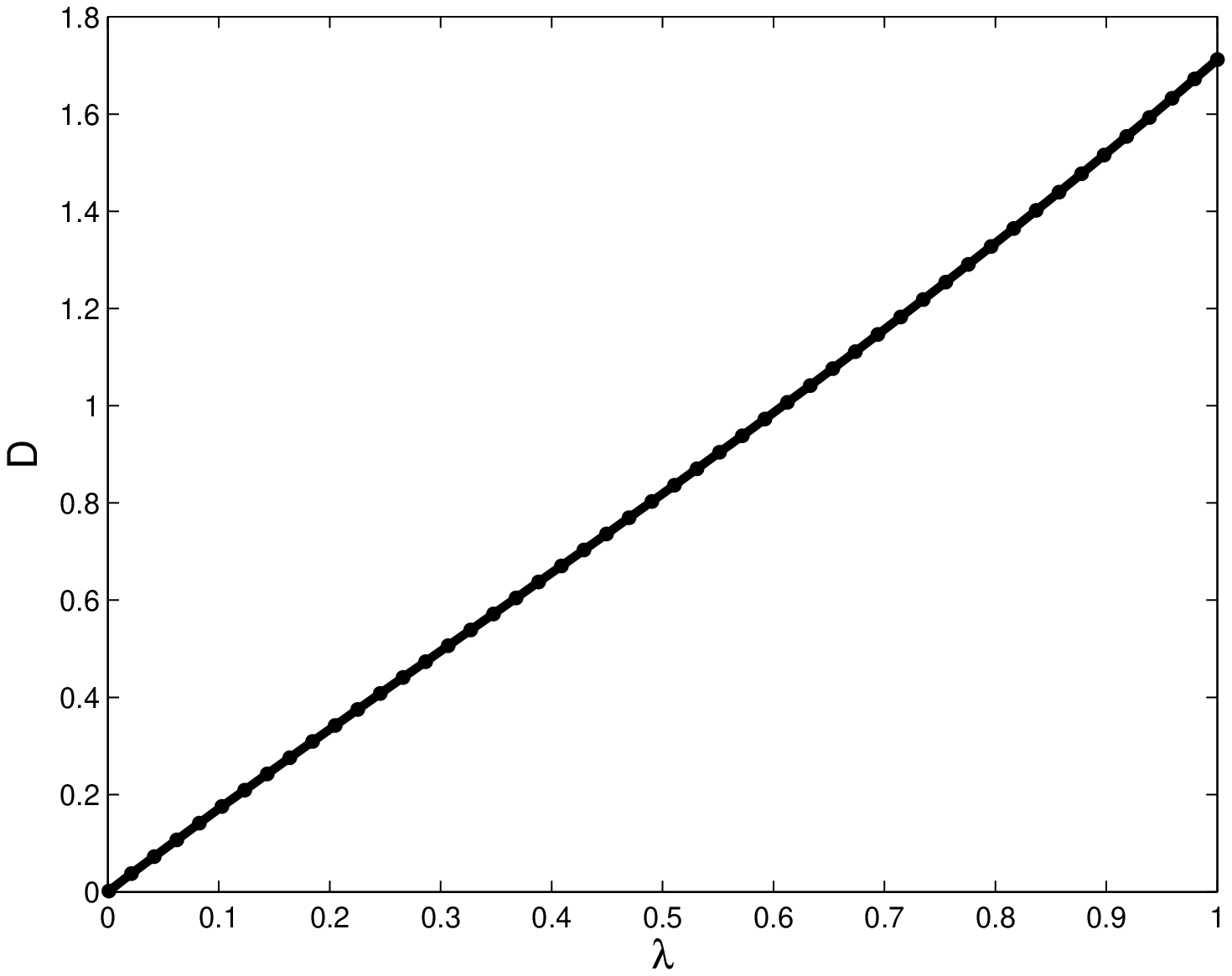} &
\includegraphics[width=5.75cm]{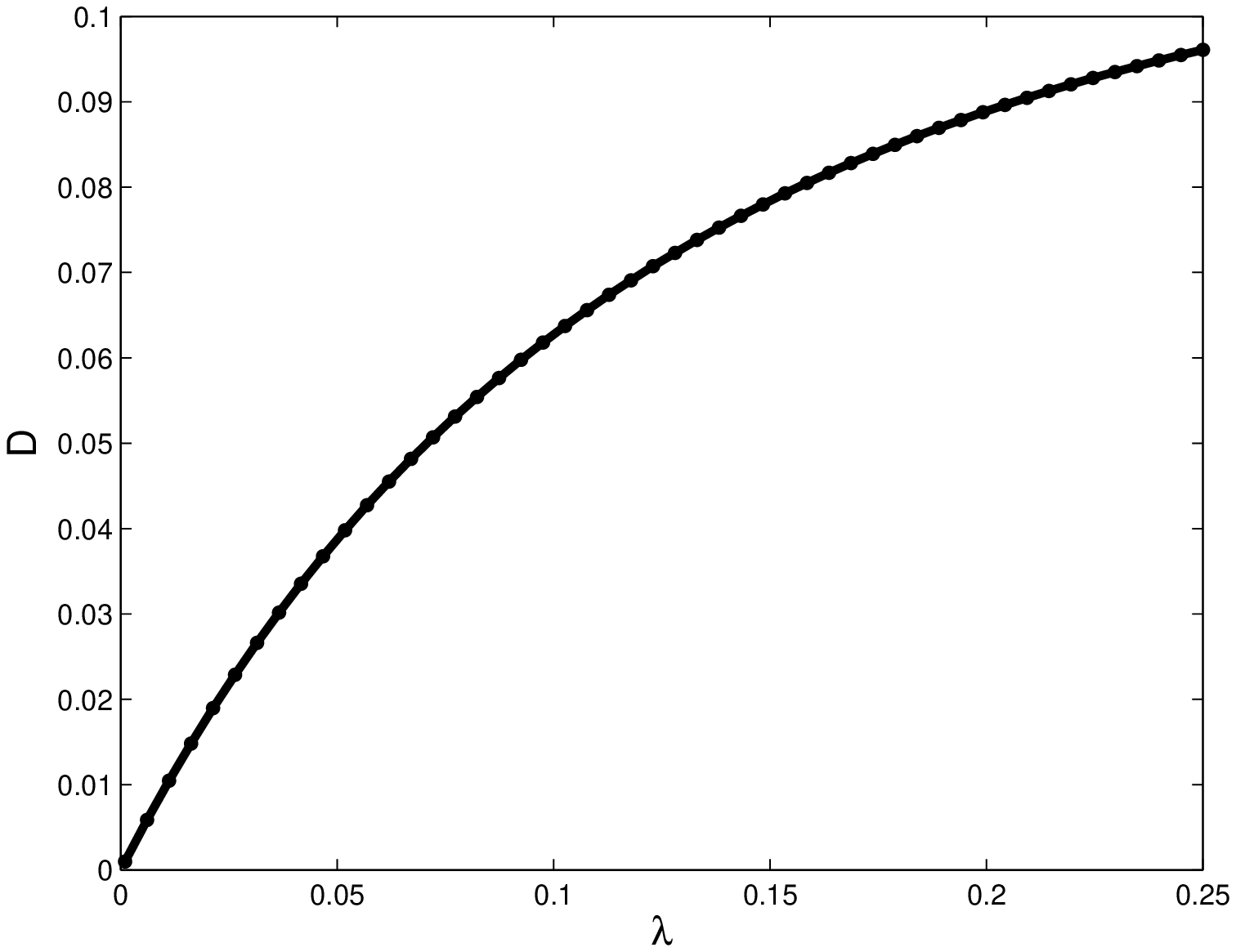} \\
\mbox{\bf (a)} & \mbox{\bf (b)}
\end{array}$
\end{center}
\caption{Typical examples of the Evans function evaluated along the
positive real axis.  The $(a)$ inflow case is computed for $v_0 = 0.7$ and
$v_0=0$ and $(b)$ the outflow case is computed for $v_0=0.3$ and
$v_+=0.001$.  Not the transversality at the origin in both cases.  Both
graphs consist of $50$ points in $\lambda$.}
\label{second2}
\end{figure}

In our study, we systematically varied $v_0$ in the interval $[.01,.99]$
and took the $v_+\rightarrow 0$ limit at each step, starting from a
$v_+=.9$ (or some other appropriate value, for example when $v_0<.9$) on
the small-amplitude end and decreased $v_+$ steadily to $10^{-k}$ for
$k=1,2,3,\ldots,6$, followed by evaluation at $v_+=0$.  For both inflow
and outflow cases, over $2000$ contours were computed.  We remark that in
the $v_+\rightarrow 0$ limit, the system becomes pressureless, and thus
all of the contours in the large-amplitude limit look the same regardless
of the value of $\gamma$ chosen.

\subsection{Nonexistence of unstable real eigenvalues}

As an additional verification of stability, we computed the Evans function
along the unstable real axis on the interval $[0,15]$ for varying
parameters to show that there are no real unstable eigenvalues.  Since the
Evans function has a root at the origin in the limiting system for the inflow case, 
and for all values of $v_+$ in the outflow case, we can perform in these cases a 
sort of  {\em numerical stability index analysis} to verify that the Evans function
cuts transversely through the origin and is otherwise nonzero, indicating
that there are no unstable real eigenvalues as expected.  In Figure
\ref{second2}, we see a typical example of $(a)$ the inflow and $(b)$
outflow cases.  Note that in both images, the Evans function cuts
transversally through the origin and is otherwise nonzero as $\lambda$
increases.

\begin{figure}[t]
\begin{center}
$\begin{array}{cc}
\includegraphics[width=5.4cm]{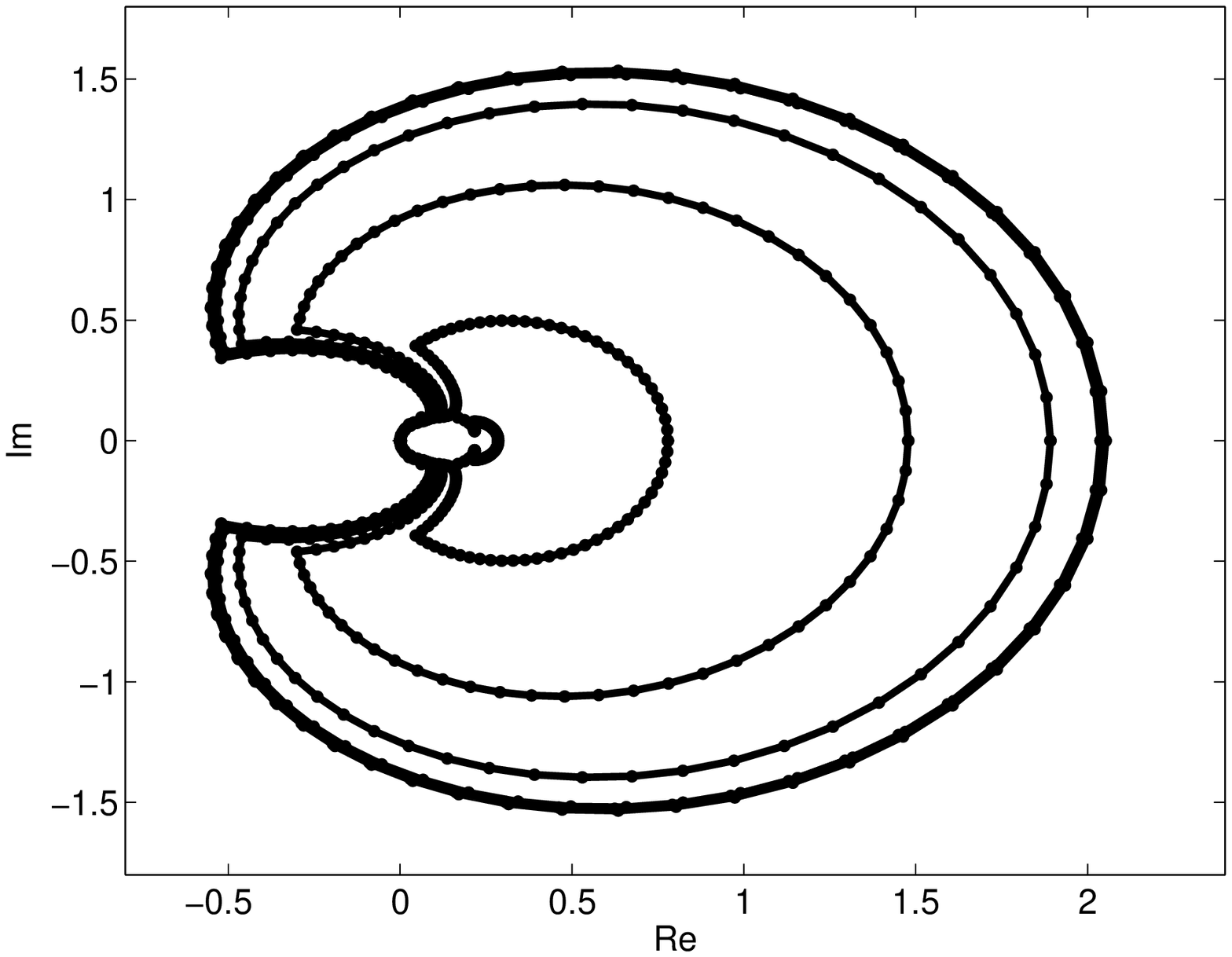} &
\includegraphics[width=5.9cm]{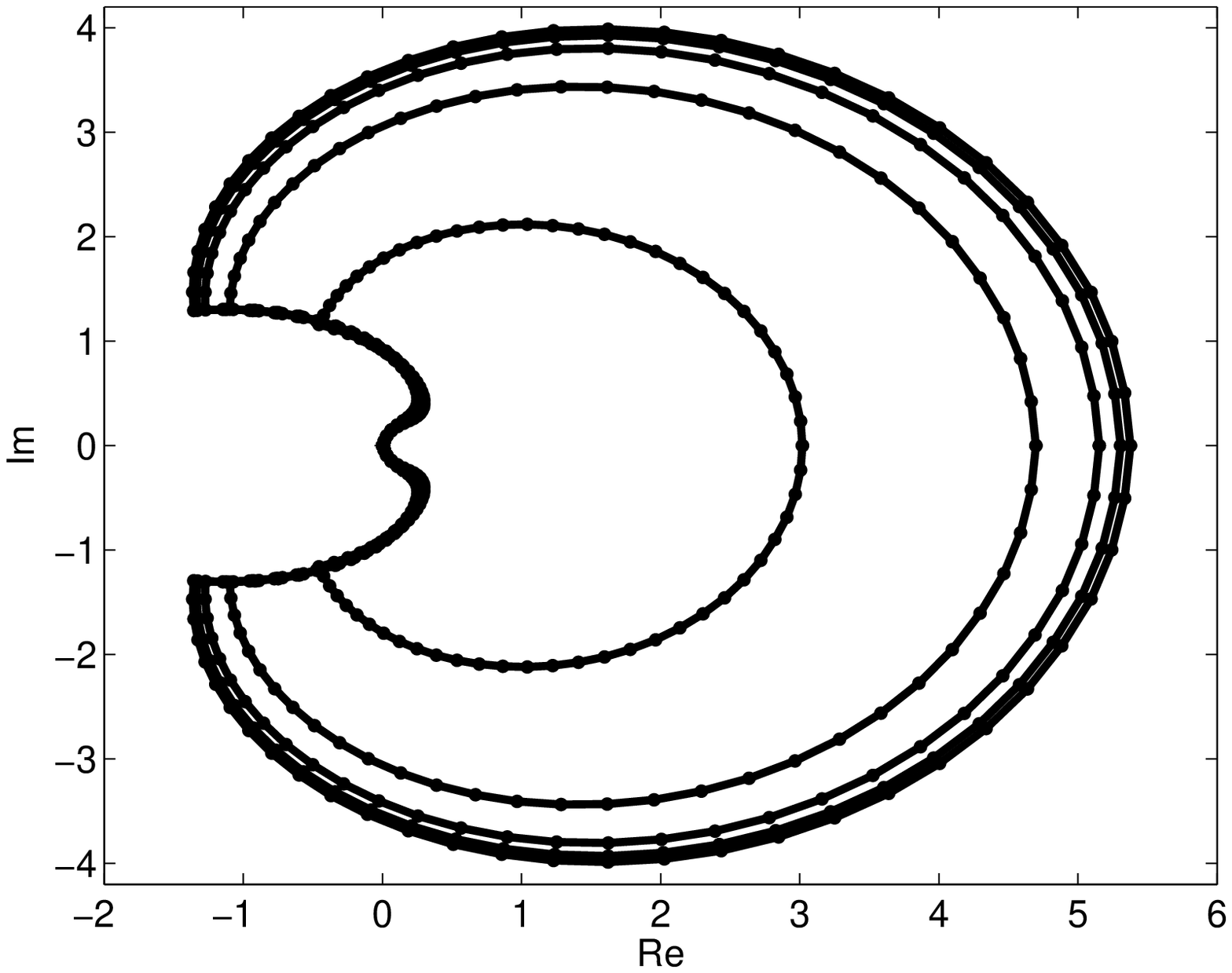} \\
\mbox{\bf (a)} & \mbox{\bf (b)}
\end{array}$
\end{center}
\caption{Shock limit for $(a)$ inflow and $(b)$ outflow cases, both for
$\gamma=5/3$.  Note that the images look very similar to those of
\cite{BHRZ,HLZ}. }
\label{second3}
\end{figure}

\subsection{The shock limit}
\label{sectionshocklimit}
When $v_0$ is far from the midpoint $(1-v_+)/2$ of the end states, the the
Evans function of the boundary layer is similar to the Evans function of
the shock case evaluated at the displacement point $x_0$.  Hence, when we
compute the boundary layer Evans function near the shock limits,
$v_0\approx 1$ for the inflow case and $v_0\approx 0$ for the outflow
case, we multiply for the correction factor $c(\lambda)$ so that our
output looks close to that of the shock case studied in \cite{BHRZ,HLZ}. 
The correction factors are
\[
c(\lambda) = e^{(-\mu^+ - \bar{\mu}^-)x_0}
\]
for the inflow case and
\[
c(\lambda) = e^{(-\bar{\mu}^+ - \mu^-)x_0},
\]
for the outflow case, where $\mu^-$ is the growth mode of $A_-(\lambda)$
and $\mu^+$ is the decay mode of $A_+(\lambda)$.  In Figure \ref{second3},
we see that these highly displaced profiles appear to be very similar to
the shock cases with one notable difference.  These images have a small
dimple near $\lambda=0$ to account for the eigenvalue there, 
whereas those in the shock case \cite{BHRZ,HLZ} were computed 
in integrated coordinates and thus have no root at the origin.

\begin{table}[t]
\begin{center}
\begin{tabular}{ccccccc}
\hline
\multicolumn{7}{c}{Inflow Case}\\
\hline
$L$ & $\gamma=1.2$ & $\gamma=1.4$ & $\gamma=1.666$ & $\gamma=2.0$ & $\gamma=2.5$ & $\gamma=3.0$ \\
\hline
8 & 7.8(-1) & 8.4(-1) & 9.2(-1) & 1.0(0) & 1.2(0) & 1.3(0) \\
10 & 1.4(-1) & 1.2(-1) & 9.2(-2) & 6.8(-2) & 4.4(-2) & 2.8(-2) \\
12 & 1.4(-2) & 7.9(-3) & 3.6(-3) & 1.3(-3) & 3.1(-4) & 7.3(-5) \\
14& 1.3(-3) & 4.9(-4) & 1.3(-4) & 2.4(-5) & 8.7(-6) & 8.2(-6)\\
16 & 1.2(-4) & 3.0(-5) & 4.7(-6) & 2.8(-6) & 2.7(-6) & 2.6(-6)\\
18& 1.1(-5) & 5.8(-6) & 8.0(-6) & 8.1(-6) & 8.0(-6) & 8.0(-6) \\
\hline
\multicolumn{7}{c}{Outflow Case}\\
\hline
$L$ & $\gamma=1.2$ & $\gamma=1.4$ & $\gamma=1.666$ & $\gamma=2.0$ & $\gamma=2.5$ & $\gamma=3.0$ \\
\hline
8 & 5.4(-3) & 5.4(-3) & 5.4(-3) & 5.4(-3) & 5.4(-3) & 5.4(-3)\\
10 & 9.2(-4) & 9.1(-4) & 9.1(-4) & 9.1(-4) & 9.1(-4) & 9.1(-4) \\
12 & 1.5(-4) & 1.5(-4) & 1.5(-4) & 1.5(-4) & 1.5(-4) & 1.5(-4) \\
14& 2.5(-5) & 2.7(-5) & 2.0(-5) & 2.0(-5) & 2.0(-5) & 2.0(-5) \\
16 & 2.3(-6) & 2.6(-6) & 2.6(-6) & 2.5(-6) & 2.5(-6) & 2.5(-6) \\
18& 6.6(-6) & 3.6(-6) & 8.7(-6) & 8.7(-6) & 8.7(-6) & 8.7(-6) \\
\hline
\end{tabular}
\caption{Relative errors in $D(\lambda)$ for the inflow and outflow cases are computed by taking the maximum relative error for 60 contour points evaluated along the semicircle $\phi$.  Samples were taken for varying $L$ and $\gamma$, leaving $v_+$ fixed at $v_+=10^{-4}$ and $v_0=0.6$.  We used $L=8,10,12,14,16,18,20$ and $\gamma=1.2,1.4,1.666,2.0$.  Relative errors were computed using the next value of $L$ as the baseline.}
\end{center}
\end{table}

\begin{table}[h]
\begin{center}
\begin{tabular}{ccccccc}
\hline
\multicolumn{7}{c}{Inflow Case}\\
\hline
\hline
Abs/Rel & $\gamma=1.2$ & $\gamma=1.4$ & $\gamma=1.666$ & $\gamma=2.0$ & $\gamma=2.5$ & $\gamma=3.0$ \\
\hline
$10^{-3}/10^{-5}$ & 5.4(-4) & 4.1(-4) & 4.0(-4) & 5.0(-4) & 3.4(-4) & 8.6(-4)\\
$10^{-4}/10^{-6}$ & 3.1(-5) & 4.6(-5) & 3.4(-5) & 3.3(-5) & 3.3(-5) & 3.2(-5)\\
$10^{-5}/10^{-7}$ & 2.9(-6) & 3.6(-6) & 3.9(-6) & 6.8(-6) & 2.7(-6) & 2.5(-6)\\
$10^{-6}/10^{-8}$ & 4.6(-7) & 9.9(-7) & 1.1(-6) & 6.0(-7) & 2.9(-7) & 3.2(-7)\\
\hline
\multicolumn{7}{c}{Outflow Case}\\
\hline
Abs/Rel & $\gamma=1.2$ & $\gamma=1.4$ & $\gamma=1.666$ & $\gamma=2.0$  & $\gamma=2.5$ & $\gamma=3.0$\\
\hline
$10^{-3}/10^{-5}$ & 9.2(-4) & 9.2(-4) & 9.1(-4) & 9.1(-4) & 9.1(-4) & 9.2(-4) \\
$10^{-4}/10^{-6}$ & 5.3(-5) & 4.9(-5) & 5.3(-5) & 5.3(-5) & 5.3(-5) & 5.3(-5)\\
$10^{-5}/10^{-7}$ & 6.7(-5) & 6.7(-5) & 6.7(-5) & 6.7(-5) & 6.7(-5) & 6.7(-5)\\
$10^{-6}/10^{-8}$ & 2.9(-6) & 2.9(-6) & 2.9(-6) & 2.9(-6) & 2.9(-6) & 2.9(-6)\\
\hline
\end{tabular}
\caption{Relative errors in $D(\lambda)$ for the inflow and outflow cases are computed by taking the maximum relative error for 60 contour points evaluated along the semicircle $\phi$.  Samples were taken for varying the absolute and relative error tolerances and $\gamma$ in the ODE solver, leaving $L=18$ and $\gamma=1.666$, $v_+=10^{-4}$, and $v_0=0.6$ fixed.  Relative errors were computed using the next run as the baseline.}
\end{center}
\end{table}

\subsection{Numerical convergence study}
\label{errorstudy}
As in \cite{BHRZ}, we carry out a numerical convergence study to show that our results are accurate.  We varied the absolute and relative error tolerances, as well as the length of the numerical domain $[-L,L]$.  In Tables 1--2, we demonstrate that our choices of $L=18$, {\tt AbsTol=1e-6} and {\tt RelTol=1e-8} provide accurate results.

\appendix

\section{Proof of preliminary estimate: inflow case}\label{basicproof}
Our starting point is Remark \ref{shockrel}, in which we observed
that the first-order eigensystem \eqref{evans_ode}
in variable $W=(w,u-v,v)^T$ may be converted by the
rescaling $W\to \tilde W:= (w,u-v, \lambda v)^T$ to a system identical
to that of the integrated equations in the shock case; see \cite{PZ}.
Artificially defining $(\tilde u, \tilde v, \tilde v')^T:= \tilde W$, we
obtain a system
\begin{subequations}\label{ep}
\begin{align}
&\lambda \tilde v + \tilde v' - \tilde u' =0, \label{ep:1}\\
&\lambda \tilde u + \tilde u' -  \frac{h(\bV)}{\bV^{\gamma+1}} \tilde v'
= \frac{\tilde u''}{\bV}.\label{ep:2}
\end{align}
\end{subequations}
identical to that in the integrated shock case \cite{BHRZ}, but with
boundary conditions
\begin {equation}\label{newbc}
\tilde v(0)=\tilde v'(0)=\tilde u'(0)=0
\end{equation}
imposed at $x=0$.
This new eigenvalue problem differs spectrally from \eqref{eigen1} only at $\lambda=0$, hence spectral stability of \eqref{eigen1} is implied by spectral stability of \eqref{ep}.
Hereafter, we drop the tildes, and refer simply to $u$, $v$.

With these coordinates, we may establish \eqref{hf} by exactly
the same argument used in the shock case in \cite{BHRZ,HLZ},
for completeness reproduced here.

\begin{lemma}
The following identity holds for $\R \lambda \geq 0$:
\begin{align}
(\R(\lambda) + |\I (\lambda)|) & \ip \bV |u|^2 + \ip |u'|^2\notag\\
 &\leq \sqrt{2} \ip \frac{h(\bV)}{\bV^\gamma} |v'| |u| +  \sqrt2\ip \bV |u'||u|\label{id1}.
\end{align}
\end{lemma}

\begin{proof}
We multiply \eqref{ep:2} by $\bV {\bar u}$ and integrate along $x$.  This yields
\[
\lambda \ip \bV |u|^2 + \ip \bV u'\bar{u} + \ip |u'|^2 = \ip \frac{h(\bV)}{\bV^\gamma} v'\bar{u} .
\]
We get \eqref{id1} by taking the real and imaginary parts and adding them together, and noting that $|\R(z)| + |\I(z)| \leq \sqrt{2}|z|$.
\end{proof}

\begin{lemma}
\label{kawashima}
The following identity holds for $\R \lambda \geq 0$:
%
%
\begin{equation}
\label{id3}
\ip |u'|^2 = 2\R(\lambda)^2\ip|v|^2 + \R(\lambda)\ip \frac{|v'|^2}{\bV} + \frac{1}{2} \ip \left[\frac{h(\bV)}{\bV^{\gamma+1}} + \frac{a\gamma}{\bV^{\gamma+1}} \right] |v'|^2
\end{equation}
\end{lemma}

\begin{proof}
We multiply \eqref{ep:2} by ${\bar v'}$ and integrate along $x$.  This yields
\[
\lambda \ip u\bar{v}' + \ip u'\bar{v}' - \ip \frac{h(\bV)}{\bV^{\gamma+1}}|v'|^2 = \ip \frac{1}{\bV}u''\bar{v}' = \ip \frac{1}{\bV}(\lambda v' + v''){\bar v'}.
\]
Using \eqref{ep:1} on the right-hand side, integrating by parts, and taking the real part gives
\[
\R \left[ \lambda \ip u\bar{v}' + \ip u'\bar{v}'\right] = \ip \left[\frac{h(\bV)}{\bV^{\gamma+1}} + \frac{\bV_x}{2 \bV^2} \right] |v'|^2 + \R(\lambda)\ip \frac{|v'|^2}{\bV}.
\]
The right hand side can be rewritten as
%
%
\begin{equation}
\label{id3_1}
\R \left[ \lambda \ip u\bar{v}' + \ip u'\bar{v}'\right] = \frac{1}{2} \ip \left[\frac{h(\bV)}{\bV^{\gamma+1}} + \frac{a\gamma}{\bV^{\gamma+1}} \right] |v'|^2 + \R(\lambda)\ip \frac{|v'|^2}{\bV}.
\end{equation}
Now we manipulate the left-hand side.  Note that
\begin{align*}
\lambda \ip u\bar{v}' + \ip u'\bar{v}' &= (\lambda+\bar{\lambda}) \ip u\bar{v}' - \ip u(\bar{\lambda}\bar{v}' + \bar{v}'')\\
&= -2\R(\lambda) \ip u' \bar{v} - \ip u \bar{u}''\\
&= -2\R(\lambda) \ip (\lambda v + v') \bar{v} + \ip |u'|^2.
\end{align*}
Hence, by taking the real part we get
\[
\R \left[ \lambda \ip u\bar{v}' + \ip u'\bar{v}'\right] = \ip |u'|^2 - 2\R(\lambda)^2 \ip |v|^2.
\]
This combines with \eqref{id3_1} to give \eqref{id3}.
\end{proof}

\begin{lemma}[\cite{BHRZ}]\label{Hlem}
For $h(\bV)$ as in \eqref{f}, we have
\begin{equation}
\label{id2}
\sup_{\bV} \left| \frac{h(\bV)}{\bV^\gamma}\right| = \gamma
\frac{1-v_+}{1-v_+^\gamma}
\leq \gamma,
\end{equation}
where $\bV$ is the profile solution to \eqref{profeq}.
\end{lemma}

\begin{proof}
Defining
\begin{equation}\label{H}
g(\bV):=h(\bV)\bV^{-\gamma} = -\bV + a(\gamma-1)\bV^{-\gamma} +
(a+1),
\end{equation}
we have $g'(\bV)= -1 -a\gamma(\gamma-1)\bV^{-\gamma-1}<0$ for
$0<v_+\le \bV\le v_-= 1$, hence the maximum of $g$ on $\bV\in
[v_+,v_-]$ is achieved at $\bV=v_+$. Substituting \eqref{RH} into
\eqref{H} and simplifying yields \eqref{id2}.
\end{proof}

\begin{proof}[Proof of Proposition \ref{hf}]
Using Young's inequality twice on right-hand side of \eqref{id1} together with \eqref{id2}, we get
\begin{align*}
(\R&(\lambda) + |\I (\lambda)|) \ip \bV |u|^2  + \ip |u'|^2 \\
&\leq \sqrt{2} \ip \frac{h(\bV)}{\bV^\gamma} |v'| |u| +  \sqrt2\ip \bV |u'||u|\\
&\leq \theta \ip \frac{h(\bV)}{\bV^{\gamma+1}} |v'|^2 + \frac{(\sqrt{2})^2}{4\theta} \ip \frac{h(\bV)}{\bV^\gamma} \bV |u|^2 + \epsilon \ip \bV |u'|^2 + \frac{1}{4 \epsilon} \ip \bV |u|^2\\
&< \theta \ip \frac{h(\bV)}{\bV^{\gamma+1}} |v'|^2  + \epsilon \ip
|u'|^2 + \left[\frac{\gamma}{2\theta} + \frac{1}{2 \epsilon}\right]
\ip \bV |u|^2.
\end{align*}
Assuming that $0<\epsilon<1$ and $\theta = (1-\epsilon)/2$, this simplifies to
\begin{align*}
(\R(\lambda) + |\I (\lambda)|) & \ip \bV |u|^2 + (1-\epsilon) \ip |u'|^2 \\
&<\frac{1-\epsilon}{2} \ip \frac{h(\bV)}{\bV^{\gamma+1}} |v'|^2 +
\left[\frac{\gamma}{2\theta} + \frac{1}{2 \epsilon}\right] \ip \bV
|u|^2.
\end{align*}
Applying \eqref{id3} yields
\[
(\R(\lambda) + |\I (\lambda)|) \ip \bV |u|^2  <
\left[\frac{\gamma}{1-\epsilon} + \frac{1}{2 \epsilon}\right]  \ip
\bV |u|^2,
\]
or equivalently,
\[
(\R(\lambda) + |\I (\lambda)|) <  \frac{(2 \gamma-1)\epsilon +
1}{2\epsilon(1-\epsilon)}.
\]
Setting $\epsilon = 1/(2\sqrt{\gamma}+1)$ gives \eqref{hfbounds1}.
\end{proof}

\section{Proof of preliminary estimate: outflow case}\label{outbasicproof}

\newcommand{\ipo}{\int_{\mathbb{R}^-}}
\newcommand{\ipi}[1]{\int_0^\infty{#1}dx}

Similarly as in the inflow case, we can convert the eigenvalue
equations into the integrated equations as in the shock case; see
\cite{PZ}. Artificially defining $(\tilde u, \tilde v, \tilde
v')^T:= \tilde W$, we obtain a system
\begin{subequations}\label{outep}
\begin{align}
&\lambda \tilde v + \tilde v' - \tilde u' =0, \label{outep:1}\\
&\lambda \tilde u + \tilde u' -  \frac{h(\bV)}{\bV^{\gamma+1}}
\tilde v' = \frac{\tilde u''}{\bV}.\label{outep:2}
\end{align}
\end{subequations}
identical to that in the integrated shock case \cite{BHRZ}, but with
boundary conditions
\begin {equation}\label{outnewbc}
\tilde v'(0) = \frac{\lambda}{\alpha-1}\tilde v(0),\quad \tilde
u'(0) = \alpha\tilde v'(0)
\end{equation}
imposed at $x=0$.
We shall write $w_0$ for $w(0)$, for any function $w$. This new
eigenvalue problem differs spectrally from \eqref{eigen1} only at
$\lambda=0$, hence spectral stability of \eqref{eigen1} is implied
by spectral stability of \eqref{outep}. Hereafter, we drop the
tildes, and refer simply to $u$, $v$.


\begin{lemma}
The following identity holds for $\R \lambda \geq 0$:
\begin{align}
(\R(\lambda) + |\I (\lambda)|) & \ipo  \bV |u|^2 - \frac{1}{2}\ipo  \bV_x |u|^2 + \ipo  |u'|^2+\frac12\bv_0|u_0|^2\notag\\
 &\leq \sqrt{2} \ipo  \frac{h(\bV)}{\bV^\gamma} |v'| |u| +  \ipo  \bV |u'||u| + \sqrt2|\alpha||v'_0||u_0|\label{outid1}.
\end{align}
\end{lemma}

\begin{proof}
We multiply \eqref{outep:2} by $\bV {\bar u}$ and integrate along
$x$. This yields
\[
\lambda \ipo  \bV |u|^2 + \ipo  \bV u'\bar{u} + \ipo  |u'|^2 = \ipo
\frac{h(\bV)}{\bV^\gamma} v'\bar{u} +u'_0\bar u_0.
\]
We get \eqref{outid1} by taking the real and imaginary parts and
adding them together, and noting that $|\R(z)| + |\I(z)| \leq
\sqrt{2}|z|$.
\end{proof}

\begin{lemma}
\label{outkawashima} The following inequality holds for $\R \lambda
\geq 0$:
%
%
\begin{align}\label{outid3} \frac{1}{2} \ipo
\left[\frac{h(\bV)}{\bV^{\gamma+1}} + \frac{a\gamma}{\bV^{\gamma+1}}
\right] |v'|^2 &+ \R(\lambda)\ipo  \frac{|v'|^2}{\bV}+ \frac
{|v'_0|^2}{4\bv_0} + 2\R\lambda ^2\ipo  |v|^2\notag \\&\le \ipo
|u'|^2 +\bv_0|u_0|^2.\end{align}

\end{lemma}

\begin{proof}
We multiply \eqref{outep:2} by ${\bar v'}$ and integrate along $x$.
This yields
\[
\lambda \ipo  u\bar{v}' + \ipo  u'\bar{v}' - \ipo
\frac{h(\bV)}{\bV^{\gamma+1}}|v'|^2 = \ipo  \frac{1}{\bV}u''\bar{v}'
= \ipo  \frac{1}{\bV}(\lambda v' + v''){\bar v'}.
\]
Using \eqref{outep:1} on the right-hand side, integrating by parts,
and taking the real part gives
\[
\R \left[ \lambda \ipo  u\bar{v}' + \ipo  u'\bar{v}'\right] = \ipo
\left[\frac{h(\bV)}{\bV^{\gamma+1}} + \frac{\bV_x}{2 \bV^2} \right]
|v'|^2 + \R(\lambda)\ipo  \frac{|v'|^2}{\bV} + \frac
{|v'_0|^2}{2\bv_0}.
\]
The right hand side can be rewritten as
%
%
\begin{align}
\label{outid3_1} \R &\left[ \lambda \ipo  u\bar{v}' + \ipo
u'\bar{v}'\right] \notag\\&= \frac{1}{2} \ipo
\left[\frac{h(\bV)}{\bV^{\gamma+1}} + \frac{a\gamma}{\bV^{\gamma+1}}
\right] |v'|^2 + \R(\lambda)\ipo  \frac{|v'|^2}{\bV}+ \frac
{|v'_0|^2}{2\bv_0}.
\end{align}
Now we manipulate the left-hand side.  Note that
\begin{align*}
\lambda \ipo  u\bar{v}' &+ \ipo  u'\bar{v}' = (\lambda+\bar{\lambda}) \ipo  u\bar{v}' + \ipo  (u'\bar{v}' - \bar \lambda u \bar v')\\
&= -2\R(\lambda) \ipo  u' \bar{v} + 2\R\lambda u_0\bar v_0+\ipo
u'(\bar v' + \bar \lambda \bar v) -\bar \lambda u_0 \bar
v_0\\
&= -2\R(\lambda) \ipo  (\lambda v + v') \bar{v} +\ipo  |u'|^2
+2\R\lambda u_0\bar v_0-\bar \lambda u_0 \bar v_0.
\end{align*}
Hence, by taking the real part  and noting that $$\R(2\R\lambda
u_0\bar v_0-\bar \lambda u_0 \bar v_0) =\R\lambda \R(u_0\bar
v_0)-\I\lambda \I(u_0 \bar v_0) = \R(\lambda u_0\bar v_0)$$  we get
\[
\R \left[ \lambda \ipo  u\bar{v}' + \ipo  u'\bar{v}'\right] = \ipo
|u'|^2 - 2\R\lambda ^2\ipo  |v|^2 - \R\lambda |v_0|^2 + \R(\lambda
u_0\bar v_0).
\]

This combines with \eqref{outid3_1} to give
\begin{align*} \frac{1}{2} \ipo
\left[\frac{h(\bV)}{\bV^{\gamma+1}} + \frac{a\gamma}{\bV^{\gamma+1}}
\right] |v'|^2 &+ \R(\lambda)\ipo  \frac{|v'|^2}{\bV}+ \frac
{|v'_0|^2}{2\bv_0} + 2\R\lambda ^2\ipo  |v|^2\\&  + \R\lambda
|v_0|^2 = \ipo  |u'|^2 +\R(\lambda u_0\bar v_0).\end{align*}

We get \eqref{outid3} by observing that \eqref{outnewbc} and Young's
inequality yield
$$|\R(\lambda u_0\bar v_0)|\le |\alpha-1||v'_0v_0| \le |v'_0v_0|\le \frac{|v'_0|^2}{4\bv_0} +
\bv_0|u_0|^2.$$ Here we used $|\alpha-1| =
\frac{|\lambda|}{|\lambda-\bv'_0|}\le 1$. Note that $\R\lambda \ge0$
and $\bv'_0\le0$.\end{proof}

\begin{proof}[Proof of Proposition \ref{hf}]
Using Young's inequality twice on right-hand side of \eqref{outid1}
together with \eqref{id2}, and denoting the boundary term on the
right by $I_b$, we get
\begin{align*}
(\R&(\lambda) + |\I (\lambda)|) \ipo  \bV |u|^2 - \frac{1}{2}\ipo  \bV_x |u|^2 + \ipo  |u'|^2+\frac12\bv_0|u_0|^2 \\
&\leq \sqrt{2} \ipo  \frac{h(\bV)}{\bV^\gamma} |v'| |u| +  \ipo  \bV |u'||u|+I_b\\
&\leq \theta \ipo  \frac{h(\bV)}{\bV^{\gamma+1}} |v'|^2 + \frac{1}{2\theta} \ipo  \frac{h(\bV)}{\bV^\gamma} \bV |u|^2 + \epsilon \ipo  \bV |u'|^2 + \frac{1}{4 \epsilon} \ipo  \bV |u|^2+I_b\\
&< \theta \ipo  \frac{h(\bV)}{\bV^{\gamma+1}} |v'|^2  + \epsilon
\ipo |u'|^2 + \left[\frac{\gamma}{2\theta} + \frac{1}{4
\epsilon}\right] \ipo  \bV |u|^2+I_b.
\end{align*}
Here we treat the boundary term by
\begin{align*}I_b&\le\sqrt2|\alpha||v'_0||u_0|\le \frac \theta 2 \frac{|v'_0|^2}{\bv_0}+\frac 1\theta |\alpha|^2 \bv_0|u_0|^2.\end{align*}

Therefore using \eqref{outid3}, we simply obtain from the above
estimates
\begin{align*}
(\R(\lambda) &+ |\I (\lambda)|)  \ipo  \bV |u|^2 + (1-\epsilon )\ipo  |u'|^2 +\frac12\bv_0|u_0|^2\\
&<\theta \ipo  \frac{h(\bV)}{\bV^{\gamma+1}} |v'|^2 +\frac \theta 2
\frac{|v'_0|^2}{\bv_0} + \left[\frac{\gamma}{2\theta} + \frac{1}{4
\epsilon}\right] \ipo  \bV |u|^2+\frac 1\theta |\alpha|^2
\bv_0|u_0|^2\\&< 2\theta \ipo  |u'|^2 +\left[\frac{\gamma}{2\theta}
+ \frac{1}{4 \epsilon}\right] \ipo  \bV |u|^2+J_b
\end{align*} where $J_b:=(\frac
1\theta |\alpha|^2 +2\theta )\bv_0|u_0|^2.$ Assuming that $\epsilon
+ 2\theta \le 1$, this simplifies to
\begin{align*}
(\R(\lambda) + |\I (\lambda)|) & \ipo  \bV |u|^2
+\frac12\bv_0|u_0|^2<\left[\frac{\gamma}{2\theta} + \frac{1}{4
\epsilon}\right] \ipo  \bV |u|^2+J_b.
\end{align*}

Note that $|\alpha|\le \frac{-\bv'_0}{|\lambda|}\le \frac
1{4|\lambda|}$. Therefore for $|\lambda|\ge \frac1{4\theta}$, we get
$|\alpha|\le \theta$ and $J_b\le 3\theta\bv_0|u_0|^2.$ For sake of
simplicity, choose $\theta =1/6$ and $\epsilon = 2/3$. This shows
that $J_b$ can be absorbed into the left by the term $\frac
12\bv_0|u_0|^2$ and thus we get
\begin{align*}
(\R(\lambda) + |\I (\lambda)|) & \ipo  \bV |u|^2 <\left[
\frac{\gamma}{2\theta}+ \frac{1}{4 \epsilon}\right] \ipo  \bV |u|^2
= \left[ 3\gamma+ \frac{3}{8}\right] \ipo  \bV |u|^2,
\end{align*} provided that $|\lambda|\ge 1/(4\theta) =3/2$.

This shows
\[
(\R(\lambda) + |\I (\lambda)|) <  \max\{\frac{3\sqrt2}{2},3\gamma+
\frac{3}{8}\}.
\]
\end{proof}

\section{Nonvanishing of $D^0_{\rm in}$}\label{stronglimit}

Working in $(\tilde v, \tilde u)$ variables as in \eqref{ep},
the limiting eigenvalue system and boundary
conditions take the form
\begin{subequations}\label{limitep}
\begin{align}
&\lambda \tilde v + \tilde v' - \tilde u' =0, \label{limitep:1}\\
&\lambda \tilde u + \tilde u' -  \frac{1-\bV}{\bV} \tilde v' =
\frac{\tilde u''}{\bV}.\label{limitep:2}
\end{align}
\end{subequations}
corresponding to a pressureless gas, $\gamma=0$,
with
\begin{equation}\label{PLBC}
(\tilde u,\tilde u',\tilde v,\tilde v')(0)=(d,0,0,0),\:\:
(\tilde u,\tilde u',\tilde v,\tilde v')(+\infty)=(c,0,0,0).
\end{equation}
Hereafter, we drop the tildes.

\begin{proof}[Proof of Proposition \ref{redenergy}]
Multiplying \eqref{limitep:2} by $\bV \bar{u}/(1-\bV)$ and integrating on
$[0,b]\subset\mathbb{R}^+$, we obtain
\[
\lambda \int^b_0 \frac{\bV}{1-\bV}|u|^2 dx +  \int^b_0 \frac{\bV}{1-\bV}
u'\bar{u} dx -  \int^b_0 v'\bar{u} dx =  \int^b_0\frac{u'' \bar{u}}{1-\bV}
dx.
\]
Integrating the third and fourth terms by parts yields
\begin{align*}
\lambda \int^b_0 \frac{\bV}{1-\bV}|u|^2 dx &+  \int^b_0 \left[
\frac{\bV}{1-\bV} + \left( \frac{1}{1-\bV} \right)'\right] u'\bar{u} dx \\
&\quad+
\int^b_0 \frac{|u'|^2}{1-\bV} dx
+  \int^b_0 v (\overline{\lambda v + v'}) dx  \\
&= \left[ v \bar{u} + \frac{u'\bar{u}}{1-\bV} \right] \Big|^b_0.\\
\end{align*}
Taking the real part, we have

\begin{align}
&\R(\lambda) \int^b_0 \left( \frac{\bV}{1-\bV}|u|^2 + |v|^2\right) dx +
\int^b_0 g(\bV) |u|^2 dx + \int^b_0 \frac{|u'|^2}{1-\bV} dx\notag\\
&\quad  = \R \left[ v \bar{u} +
\frac{u'\bar{u}}{1-\bV} - \frac{1}{2}\left[\frac{\bV}{1-\bV} + \left(
\frac{1}{1-\bV} \right)'\right] |u|^2 - \frac{|v|^2 }{2} \right] \Big|^b_0,\label{MNen}
\end{align}
where
\[
g(\bV) =  -\frac{1}{2}\left[ \left(\frac{\bV}{1-\bV}\right)' + \left(
\frac{1}{1-\bV} \right)''\right].
\]
Note that
\[
\frac{d}{dx}\left(\frac{1}{1-\bV}\right) = - \frac{(1-\bV)'}{(1-\bV)^2} =
\frac{\bV_x}{(1-\bV)^2} = \frac{\bV(\bV-1)}{(1-\bV)^2} =
-\frac{\bV}{1-\bV}.
\]
Thus, $g(\bV)\equiv 0$ and the third term on the right-hand side vanishes,
leaving
\begin{align*}
\R(\lambda) \int^b_0 \left( \frac{\bV}{1-\bV}|u|^2 + |v|^2\right) dx &+
\int^b_0 \frac{|u'|^2}{1-\bV} dx\\
&\quad  = \left[ \R(v \bar{u}) + \frac{\R(u'\bar{u})}{1-\bV} - \frac{|v|^2
}{2} \right] \Big|^b_0\\
&\quad  = \left[ \R(v \bar{u}) + \frac{\R(u'\bar{u})}{1-\bV} - \frac{|v|^2
}{2} \right](b).\\
\end{align*}

We show finally that the right-hand side goes to zero in the limit as
$b\rightarrow\infty$.
By Proposition \ref{conjugation}, the behavior of $u$, $v$ near
$\pm \infty$ is governed by the limiting constant--coefficient
systems $W'=A^0_\pm(\lambda)W$, where $W=(u,v,v')^T$
and $A^0_\pm =A^0(\pm \infty, \lambda)$.
In particular, solutions $W$ asymptotic to $(1,0,0)$ at
$x=+\infty$ decay exponentially in $(u',v,v')$ and are bounded
in coordinate $u$ as $x\to +\infty$.
Observing that $1-\hat v\to 1$ as $x\to +\infty$, we thus see immediately
that the boundary contribution at $b$ vanishes as $b\to +\infty$.

Thus, in the limit as $b\to +\infty$,
\begin{equation}
\R(\lambda) \int^{+\infty}_{0} \left( \frac{\bV}{1-\bV}|u|^2 + |v|^2\right) dx
+ \int^{+\infty}_{0} \frac{|u'|^2}{1-\bV} dx=0.
\end{equation}
But, for $\R \lambda\ge 0$, this implies $u'\equiv 0$,
or $u\equiv \hbox{\rm constant}$, which, by $u(0)=1$,
implies $u\equiv 1$.
This reduces \eqref{limitep:1} to $v'=\lambda v$, yielding the
explicit solution $v=Ce^{\lambda x}$. By $v(0)=0$,
therefore, $v\equiv 0$ for $\R \lambda\ge 0$.
Substituting into \eqref{limitep:2}, we obtain $\lambda =0$.
It follows that there are no nontrivial solutions of \eqref{limitep},
\eqref{PLBC} for $\R \lambda\ge 0$ except at $\lambda=0$.
%
\end{proof}

\begin{remark}\label{symmetrizers}
The above energy estimate is essentially identical to that used
in \cite{HLZ} to treat the limiting shock case.
\end{remark}

\section{Nonvanishing of $D^0_{\rm out}$}\label{outnonv}

Working in $(\tilde v, \tilde u)$ variables as in \eqref{ep},
the limiting eigenvalue system and boundary
conditions take the form
\begin{subequations}\label{outlimitep}
\begin{align}
&\lambda \tilde v + \tilde v' - \tilde u' =0, \label{outlimitep:1}\\
&\lambda \tilde u + \tilde u' -  \frac{1-\bV}{\bV} \tilde v' =
\frac{\tilde u''}{\bV}.\label{outlimitep:2}
\end{align}
\end{subequations}
corresponding to a pressureless gas, $\gamma=0$,
with
\begin{equation}\label{outPLBC}
(\tilde u,\tilde u',\tilde v,\tilde v')(-\infty)=(0,0,0,0),
\end{equation}
\begin {equation}
\tilde v'(0) = \frac{\lambda}{\alpha-1}\tilde v(0),\quad \tilde
u'(0) = \alpha\tilde v'(0).
\end{equation}
In particular, 
\begin{equation}\label{useful}
\tilde u'(0)=\frac{\lambda \alpha}{\alpha -1}\tilde v(0)=
\hat v'(0)\tilde v(0)
=
(v_0-1)\hat v_0 \tilde v(0).
\end{equation}
Hereafter, we drop the tildes.

\begin{proof}[Proof of Proposition \ref{redenergy}]
Multiplying \eqref{outlimitep:2} by $\bV \bar{u}/(1-\bV)$ and integrating on
$[a,0]\subset\mathbb{R}^-$, we obtain
\[
\lambda \int^0_a \frac{\bV}{1-\bV}|u|^2 dx +  \int^b_a \frac{\bV}{1-\bV}
u'\bar{u} dx -  \int^0_a v'\bar{u} dx =  \int^0_a\frac{u'' \bar{u}}{1-\bV}
dx.
\]
Integrating the third and fourth terms by parts yields
\begin{align*}
\lambda \int^0_a \frac{\bV}{1-\bV}|u|^2 dx &+  \int^0_a \left[
\frac{\bV}{1-\bV} + \left( \frac{1}{1-\bV} \right)'\right] u'\bar{u} dx \\
&\quad+
\int^0_a \frac{|u'|^2}{1-\bV} dx
+  \int^0_a v (\overline{\lambda v + v'}) dx  \\
&= \left[ v \bar{u} + \frac{u'\bar{u}}{1-\bV} \right] \Big|^0_a.\\
\end{align*}
Taking the real part, we have

\begin{align}
&\R(\lambda) \int^0_a \left( \frac{\bV}{1-\bV}|u|^2 + |v|^2\right) dx +
\int^0_a g(\bV) |u|^2 dx + \int^0_a \frac{|u'|^2}{1-\bV} dx\notag\\
&\quad  = \R \left[ v \bar{u} +
\frac{u'\bar{u}}{1-\bV} - \frac{1}{2}\left[\frac{\bV}{1-\bV} + \left(
\frac{1}{1-\bV} \right)'\right] |u|^2 - \frac{|v|^2 }{2} \right] \Big|^0_a,\label{MNen2}
\end{align}
where
\[
g(\bV) =  -\frac{1}{2}\left[ \left(\frac{\bV}{1-\bV}\right)' + \left(
\frac{1}{1-\bV} \right)''\right]\equiv 0
\]
and the third term on the right-hand side vanishes,
as shown in Section \ref{stronglimit}, leaving
\begin{align*}
\R(\lambda) \int^0_a \left( \frac{\bV}{1-\bV}|u|^2 + |v|^2\right) dx &+
\int^0_a \frac{|u'|^2}{1-\bV} dx\\
&\quad  = \left[ \R(v \bar{u}) + \frac{\R(u'\bar{u})}{1-\bV} - \frac{|v|^2
}{2} \right] \Big|^0_a.\\
\end{align*}

A boundary analysis similar to that of Section \ref{stronglimit}
shows that the contribution at $a$ on the righthand side
vanishes as $a\to -\infty$; see \cite{HLZ} for details.
Thus, in the limit as $a\to -\infty$ we obtain
\begin{align*}
\R(\lambda) \int^0_{-\infty} \left( \frac{\bV}{1-\bV}|u|^2 + |v|^2\right) dx &+
\int^0_{-\infty} \frac{|u'|^2}{1-\bV} dx\\
&\quad  = \left[ \R(v \bar{u}) + \frac{\R(u'\bar{u})}{1-\bV} - \frac{|v|^2
}{2} \right] (0)\\
&\quad  = \left[(1-v_0)\R(v \bar{u}) - \frac{|v|^2
}{2} \right] (0),\\
&\quad  \le
\left[(1- v_0) |v||u| - \frac{|v|^2 }{2}\right](0)\\
&\quad  \le (1- v_0 )^2\frac{|u(0)|^2}{2},
\end{align*}
where the second equality follows by \eqref{useful} and the
final line by Young's inequality.

Next, observe the Sobolev-type bound
$$
\begin{aligned}
|u(0)|^2 &\le \Big(
\int_{-\infty}^0 |u'(x)| dx
\Big)^2
\le
\int_{-\infty}^0 \frac{|u'|^2}{1-\hat v}(x) dx
\int_{-\infty}^0 (1-\hat v)(x) dx,
\end{aligned}
$$
together with
$$
\begin{aligned}
\int_{-\infty}^0 (1-\hat v)(x) dx&=
\int_{-\infty}^0 -\frac{\hat v'}{\hat v}(x)dx
=\int_{-\infty}^0 (\log \hat v^{-1})'(x) dx
= \log v_0^{-1},
\end{aligned}
$$
hence
$\int_{-\infty}^0 (1-\hat v)(x) dx< \frac{2}{(1-v_0)^2}$
for $v_0>v_*$, where $v_*<e^{-2}$ is the unique solution of
\begin{equation}\label{fnl}
 v_*=e^{-2/(1-v_*)^2}. 
\end{equation}

Thus, for $v_0>v_*$,
\begin{equation}
\R(\lambda) \int^{0}_{-\infty} \left( \frac{\bV}{1-\bV}|u|^2 + |v|^2\right) dx
+ \epsilon \int^{0}_{-\infty} \frac{|u'|^2}{1-\bV} dx\le 0,
\end{equation}
for $\epsilon:= \frac{(1-v_0)^2}{2}- 
\frac{1}{\int_{-\infty}^0 (1-\hat v)(x) dx}>0$.
For $\R \lambda\ge 0$, this implies $u'\equiv 0$,
or $u\equiv \hbox{\rm constant}$, which, by $u(-\infty)=0$,
implies $u\equiv 0$.
This reduces \eqref{outlimitep:1} to $v'=\lambda v$, yielding the
explicit solution $v=Ce^{\lambda x}$. By $v(0)=0$,
therefore, $v\equiv 0$ for $\R \lambda\ge 0$.
It follows that there are no nontrivial solutions of \eqref{outlimitep},
\eqref{outPLBC} for $\R \lambda\ge 0$ except at $\lambda=0$.

By iteration, starting with $v_*\approx 0$, we obtain
first $v_*<e^{-2}\approx 0.14$ then $v_*> e^{2/(1-.14)^2}\approx .067$,
then $v_*< e^{2/(1-.067)^2}\approx .10$, then 
$v_*> e^{2/(1-.10)^2}\approx .085$, then 
$v_*< e^{2/(1-.085)}\approx .091$ and
$v_*> e^{2/(1-.091)}\approx .0889$, terminating with $v_*\approx .0899$.
\end{proof}

\begin{remark}\label{smallcase}
Our Evans function results show that
the case $v_0$ small not treated corresponds to the shock
limit for which stability is already known by \cite{HLZ}.
This suggests that a more sophisticated energy estimate
combining the above with a boundary-layer analysis from 
$x=0$ back to $x=L+\delta$ might yield nonvanishing for all $1>v_0>0$.
\end{remark}

\section{The characteristic limit: outflow case}\label{char}

We now show stability of compressive outflow boundary layers in the
characteristic limit $v_+\to 1$, by essentially the same energy
estimate used in \cite{MN} to show stability of small-amplitude
shock waves.

As in the above section on the outflow case, we obtain a system
\begin{subequations}\label{charep}
\begin{align}
&\lambda \tilde v + \tilde v' - \tilde u' =0, \label{charep:1}\\
&\lambda \tilde u + \tilde u' -  \frac{h(\bV)}{\bV^{\gamma+1}}
\tilde v' = \frac{\tilde u''}{\bV}.\label{charep:2}
\end{align}
\end{subequations}
identical to that in the integrated shock case \cite{BHRZ}, but with
boundary conditions
\begin {equation}
\tilde v'(0) = \frac{\lambda}{\alpha-1}\tilde v(0),\quad \tilde
u'(0) = \alpha\tilde v'(0).
\end{equation}
In particular,
\begin{equation}
\tilde u'(0)=\frac{\lambda \alpha}{\alpha -1}\tilde v(0)= \hat
v'(0)\tilde v(0).
\end{equation}
This new eigenvalue problem differs spectrally from \eqref{eigen1}
only at $\lambda=0$, hence spectral stability of \eqref{eigen1} is
implied by spectral stability of \eqref{charep}. Hereafter, we drop
the tildes, and refer simply to $u$, $v$.

\begin{proof} [Proof of Proposition \ref{charsmallamp}]

\newcommand{\ipco}[1]{\int_{-\infty}^0 {#1} dx}

We note that $h(\bV) > 0$.  By multiplying \eqref{charep:2} by both
the conjugate $\bar{u}$ and $\bV^{\gamma+1}/h(\bV)$ and integrating
along $x$ from $-\infty$ to $0$, we have
\[
\ipco{ \frac{\lambda u \bar{u}\bV^{\gamma+1}}{h(\bV)} }+ \ipco{
\frac{u' \bar{u}\bV^{\gamma+1}}{h(\bV)} }-  \ipco{ v' \bar{u}} =
\ipco{ \frac{u''\bar{u}\bV^\gamma}{h(\bV)}}.
\]
Integrating the last three terms by parts and appropriately using
\eqref{charep:1} to substitute for $u'$ in the third term gives us
\begin{align*}
\ipco{ \frac{\lambda |u|^2 \bV^{\gamma+1}}{h(\bV)} }&+ \ipco{
\frac{u' \bar{u}\bV^{\gamma+1}}{h(\bV)}} + \ipco{ v
(\overline{\lambda v + v'})} + \ipco{
\frac{\bV^\gamma|u'|^2}{h(\bV)}}
\\&= -\ipco{ \left(\frac{\bV^\gamma}{h(\bV)}\right)' u'\bar{u}} +
\left[v\bar u+\frac{v^\gamma u'\bar u}{h(\bv)}\right]\Big|_{x=0}.
\end{align*}
We take the real part and appropriately integrate by parts to get
\begin{align}\label{outid2}
\R(\lambda)\ipco{ \left[ \frac{\bV^{\gamma+1}}{h(\bV)}|u|^2+|v|^2
\right]} +  \ipco{ g(\bV) |u|^2} + \ipco{
\frac{\bV^\gamma}{h(\bV)}|u'|^2}= G(0),
\end{align}
where
\[
g(\bV) = - \frac{1}{2}
\left[\left(\frac{\bV^{\gamma+1}}{h(\bV)}\right)' +
\left(\frac{\bV^\gamma}{h(\bV)}\right)'' \right]
\]
and
\[
G(0) = - \frac{1}{2} \left[\frac{\bV^{\gamma+1}}{h(\bV)} +
\left(\frac{\bV^\gamma}{h(\bV)}\right)' \right]|u|^2 + \R\left[v\bar
u+\frac{v^\gamma u'\bar u}{h(\bv)}\right] - \frac{|v|^2}{2}
\]
evaluated at $x=0$. Here, the boundary term appearing on the
righthand side is the only difference from the corresponding
estimate appearing in the treatment of the shock case in \cite{MN,
BHRZ}. We shall show that as $\bv_+\to1$, the boundary term $G(0)$
is nonpositive. Observe that boundary conditions yield
$$\left[v\bar u+\frac{v^\gamma u'\bar u}{h(\bv)}\right]\Big|_{x=0} =
\R(v(0)\bar u(0))\left[1+\frac{\bv^\gamma
\bv'}{h(\bv)}\right]\Big|_{x=0}.$$

We first note, as established in \cite{MN,BHRZ}, that $g(\bV) \geq
0$ on $[v_+,1]$, under certain conditions including the case
$\bv_+\to1$. Straightforward computation gives identities:
\begin{align}
\gamma h(\bV) - \bV h'(\bV) &= a\gamma(\gamma-1) + \bV^{\gamma+1}\quad\mbox{and}\label{charell1}\\
\bV^{\gamma-1}\bV_x &= a\gamma - h(\bV)\label{charI2}.
\end{align}
Using \eqref{charell1} and \eqref{charI2}, we abbreviate a few
intermediate steps below:
\begin{align}
g(\bV) &= -\frac{\bV_x}{2}\left[ \frac{(\gamma+1)\bV^\gamma h(\bV) - \bV^{\gamma+1}h'(\bV)}{h(\bV)^2} + \frac{d}{d\bV}\left[ \frac{\gamma \bV^{\gamma-1}h(\bV)-\bV^\gamma h'(\bV)}{h(\bV)^2} \bV_x \right]\right]\notag\\
&= -\frac{\bV_x}{2}\left[ \frac{\bV^\gamma\left((\gamma+1)h(\bV) - \bV h'(\bV)\right)}{h(\bV)^2} + \frac{d}{d\bV}\left[ \frac{\gamma h(\bV)-\bV h'(\bV)}{h(\bV)^2} (a\gamma-h(\bV)) \right]\right]\notag\\
&=-\frac{a\bV_x\bV^{\gamma-1}}{2 h(\bV)^3} \times\notag\\
& \qquad\left[ \gamma^2(\gamma+1)\bV^{\gamma+2} - 2 (a+1)\gamma(\gamma^2-1)\bV^{\gamma+1}+(a+1)^2\gamma^2(\gamma-1)\bV^\gamma\right.\notag\\
&\qquad\qquad +\left. a\gamma(\gamma+2)(\gamma^2-1)\bV-a (a+1) \gamma^2 (\gamma^2-1) \right]\notag\\
&=-\frac{a\bV_x\bV^{\gamma-1}}{2 h(\bV)^3}[(\gamma+1)\bV^{\gamma+2}+\bV^\gamma(\gamma-1)\left((\gamma+1)\bV-(a+1)\gamma\right)^2 \label{charpreMN}\\
&\qquad + a\gamma(\gamma^2-1)(\gamma+2)\bV-a (a+1)\gamma^2(\gamma^2-1)]\notag\\
&\geq -\frac{a\bV_x\bV^{\gamma-1}}{2 h(\bV)^3}[(\gamma+1)\bV^{\gamma+2}+ a\gamma(\gamma^2-1)(\gamma+2)\bV-a (a+1)\gamma^2(\gamma^2-1)]\notag\\
&\geq-\frac{\gamma^2 a^3 \bV_x (\gamma+1)}{2 h(\bV)^3
v_+}\left[\left(\frac{v_+^{\gamma+1}}{a\gamma}\right)^2+2(\gamma-1)\left(\frac{v_+^{\gamma+1}}{a\gamma}\right)-(\gamma-1)\right].\label{charMN}
\end{align}
This verifies $g(\bv) \ge 0$ as $\bv_+\to1$.

Second, examine
\[ G(0) = - \frac{1}{2} \left[\frac{\bV^{\gamma+1}}{h(\bV)} +
\left(\frac{\bV^\gamma}{h(\bV)}\right)'
\right]|u(0)|^2+\left[1+\frac{\bv^\gamma
\bv'}{h(\bv)}\right]\R(v(0)\bar u(0))- \frac{|v(0)|^2}{2}.
\]

Applying Young's inequality to the middle term, we easily get $$G(0)
\le  - \frac{1}{2} \left[\frac{\bV^{\gamma+1}}{h(\bV)} +
\left(\frac{\bV^\gamma}{h(\bV)}\right)' -\left(1+\frac{\bv^\gamma
\bv'}{h(\bv)}\right)^2\right]|u(0)|^2 =: -\frac 12 I|u(0)|^2.$$

Now observe that $I$ can be written as
\begin{align*} I=\frac{\bV^{\gamma+1}}{h(\bV)} &-1+
\left[\frac{\gamma\bv^{\gamma-1}}{h(\bv)}-\frac{2\bv^\gamma}{h(\bv)}
-\frac{\bv^{2\gamma}\bv'}{h^2(\bv)}\right]\bv' - \frac{\bv^\gamma
h'(\bv)}{h^2(\bv)}.\end{align*}

Using \eqref{charell1} and \eqref{charI2}, we get
$$\frac{\bV^{\gamma+1}}{h(\bV)} -1 = -\frac{(\gamma-1)\bv^{\gamma-1}\bv' + \bv
h'(\bv)}{h(\bv)}$$ and thus
\begin{align*} I= -\frac{(\gamma-1)\bv^{\gamma-1}\bv' + \bv
h'(\bv)}{h(\bv)}+
\left[\frac{\gamma\bv^{\gamma-1}}{h(\bv)}-2\frac{\bv^\gamma}{h(\bv)}
-\frac{\bv^{2\gamma}\bv'}{h^2(\bv)}\right]\bv' - \frac{\bv^\gamma
h'(\bv)}{h^2(\bv)}.\end{align*}

Now since $h'(\bv) =
-(\gamma+1)\bv^\gamma\bv'+(a+1)\gamma\bv^{\gamma-1}\bv'$, as
$\bv_+\to 1$, $I\sim -\bv' \ge0$. Therefore, as $\bv_+$ is close to
$1$, $G(0)\le \frac 14\bv'(0)|u(0)|^2\le0$. This, $g(\bv)\ge0$, and
\eqref{outid2} give, as $\bv_+$ is close enough to $1$,
\begin{align}\label{outid4}
\R(\lambda)\ipco{ \left[ \frac{\bV^{\gamma+1}}{h(\bV)}|u|^2+|v|^2
\right]} &+\ipco {\frac{\bV^\gamma}{h(\bV)}|u'|^2}\le 0,
\end{align} which evidently gives stability as claimed.
\end{proof}

\section{Nonvanishing of $D_{\rm in}$: expansive inflow case}\label{nonvanish-expansive-inflow}

For completeness, we recall the argument of \cite{MN.2}
in the expansive inflow case.

{\bf Profile equation.} Note that, in the expansive inflow case, we
assume $v_0 < v_+$. Therefore we can still follow the scaling
\eqref{scaling} to get
$$0<v_0< v_+ =1.$$

Then the stationary boundary layer  $(\bv,\bu)$ satisfies
\eqref{stationarybl} with $v_0 <v_+=1$. Now by integrating
\eqref{scalarode} from $x$ to $+\infty$ with noting that
$\bv(+\infty)=1$ and $\bv'(+\infty)=0$, we get the profile equation
$$\bv'= \bv(\bv-1 + a(\bv^{-\gamma} - 1)).$$

Note that $\bv'>0$. We now follow
the same method for compressive inflow case to get the following
eigenvalue system
\begin{subequations}\label{expinep}
\begin{align}
&\lambda  v +  v' -  u' =0, \label{expinep:1}\\
&\lambda  u +  u' -  (f v)' = \left(\frac{
u'}{\bV}\right)'.\label{expinep:2}
\end{align}
\end{subequations}with boundary
conditions
\begin{equation} u(0) =  v(0) =0,
\end{equation} where $f(\bv) = \frac{h(\bv)}{\bv^{\gamma+1}}$.

\begin{proof}[Proof of Proposition \ref{expansive}] Multiply the equation \eqref{expinep:2} by $\bar u$
and integrate along $x$. By integration by parts, we get $$\lambda
\ipi{|u|^2} + \ipi{u'\bar u + fv\bar u' + \frac{|u'|^2}{\bv}} =0.$$

Using \eqref{expinep:1} and taking the real part of the above yield
\begin{align}\label{exeq}\R\lambda \ipi{|u|^2+f|v|^2} -\frac 12 \ipi{f'|v|^2} +
\ipi{\frac{|u'|^2}{\bv}} =0.\end{align}

Note that $$f' = \left(1+a+\frac{a(\gamma^2-1)}{\bv^\gamma}\right)
\frac{-\bv'}{\bv^2} \le 0$$ which together with \eqref{exeq} gives
$\R\lambda <0$, the proposition is proved. \end{proof}

\def\cprime{$'$}

\end{document}